\documentclass[10pt]{amsart}
\usepackage[margin=1.in]{geometry}
\usepackage{colortbl}
\newtheorem{theorem}{Theorem}[section]

\newtheorem{proposition}{Proposition}[section]

\theoremstyle{definition}

\theoremstyle{remark}
\newtheorem{remark}[theorem]{Remark}

\numberwithin{equation}{section}

\usepackage{graphicx}
\usepackage{subcaption}
\usepackage{multirow}
\usepackage{bbm}
\usepackage[bottom]{footmisc}

\pdfoutput=1
\usepackage[pagewise]{lineno}
\usepackage{enumitem}
\theoremstyle{definition}

\usepackage{xcolor}
\usepackage[export]{adjustbox}
\usepackage{caption}
\usepackage{hyperref}
\usepackage{epstopdf}
\usepackage{epsfig}

\begin{document}

\title[A SL scheme for ES-BGK model]{A conservative semi-Lagrangian scheme for the ellipsoidal BGK model of the Boltzmann equation}

\author[S. Boscarino]{Sebastiano Boscarino}
\address{Sebastiano Boscarino\\
Department of Mathematics and Computer Science  \\
University of Catania\\
95125 Catania, Italy} \email{boscarino@dmi.unict.it}

\author[S. Y. Cho]{Seung Yeon Cho}
\address{Seung Yeon Cho\\
Department of Mathematics\\
Gyeongsang National University\\ 
52828 Jinju, Republic of Korea}
\email{chosy89@gnu.ac.kr}

\author[G. Russo]{Giovanni Russo}
\address{Giovanni Russo\\
Department of Mathematics and Computer Science  \\
University of Catania\\
95125 Catania, Italy} \email{russo@dmi.unict.it}

\author[S.-B. Yun]{Seok-Bae Yun}
\address{Seok-Bae Yun\\
Department of Mathematics\\ 
Sungkyunkwan University\\
 Suwon 440-746, Republic of Korea}
\email{sbyun01@skku.edu}

\begin{abstract}

In this paper, we propose a high order conservative semi-Lagrangian scheme (SL) for the ellipsoidal BGK model of the Boltzmann transport equation. To avoid the time step restriction induced by the convection term, we adopt the semi-Lagrangian approach. For treating the nonlinear stiff relaxation operator with small Knudsen number, we employ high order $L$-stable diagonally implicit Runge-Kutta time discretization or backward difference formula. The proposed implicit schemes are designed to update solutions explicitly without resorting to any Newton solver. We present several numerical tests to demonstrate the accuracy and efficiency of the proposed methods. These methods allow us to obtain accurate approximations of the solutions to the Navier-Stokes equations or the Boltzmann equation for moderate or relatively large Knudsen numbers, respectively.

\end{abstract}
\maketitle
%\tableofcontents

\footnotetext{}

\section{Introduction}
%\linenumbers

When describing the motions of rarefied gases in the high altitude, we frequently encounter the situation where gases are in the thermal non-equilibrium. In the situation, continuum equations such as Navier-Stokes equations (NSEs) fail to capture the correct behavior of gases. To address this issue, kinetic models such as the Boltzmann transport equation (BTE) have been widely adopted for predicting the behavior of gases in such regimes. The equation reads
\begin{align}\label{BTE}
	\begin{split}
		\frac{\partial{f}}{\partial{t}} + v \cdot \nabla_x{f} &=Q(f,f),
	\end{split}
\end{align}
where $f=f(x,v,t)$ is the velocity distribution function defined on phase point $(x,v)\in\mathbb{R}^3\times \mathbb{R}^3$ at time $t$ and $Q(f,f)$ is the Boltzmann collision operator which takes the form of five-fold integral. In order to reduce the computational costs from the Boltzmann collision operator, there have been numerous techniques. Fourier-Galerkin method is one of the most popular approach \cite{MP,PP,RP}.

On the other hand, instead of treating BTE directly, simpler approximation models for BTE have been proposed as an alternative by replacing the Boltzmann collision operator with relaxation-type operators. The BGK model \cite{BGK} is the most well-known model, which substitute the Boltzmann collision operator with a relaxation towards local equilibrium: 
\begin{align}\label{BGK}
	\begin{split}
		\frac{\partial{f}}{\partial{t}} + v \cdot \nabla_x{f} &= \frac{\tau}{\varepsilon}\left(\mathcal{M}(f)-f\right),
	\end{split}
\end{align}
where $\mathcal{M}(f)$ is so-called local Maxwellian. The relaxation time $\tau$ depends on macroscopic quantities $\rho$ and $T$, and the parameter $\varepsilon$ is the Knudsen number defined as the ratio between the mean free path of gas molecules and the physical length scale of the problem of interest. Although the BGK model still maintains important properties of the BTE collision operator, its main defect is that the BGK solution fails to reproduce the correct Prandtl number (Pr) of monatomic gases. The Prandtl number for the BGK model is fixed by one, however, that of monatomic gases such as Helium or Argon is close to $2/3$, which is also consistent to the one derived from BTE for hard sphere molecules.
%The Prandtl number is defined as the ratio between the viscosity and the heat conductivity.

To treat this issue, a generalized version of the BGK model has been suggested by Holway \cite{H} by replacing the local Maxwellian with an ellipsoidal Gaussian that has a free parameter $-1/2\leq v<1$. This model is called the ellipsoidal BGK model (ES-BGK model):
\begin{align}\label{A-1}
\begin{split}
\frac{\partial{f}}{\partial{t}} + v \cdot \nabla_x{f} &= \frac{\tau}{\varepsilon}\left(\mathcal{G}(f)-f\right),
%\cr f(x,v,0) &= f_0(x,v).
\end{split}
\end{align}
where
\begin{align}\label{conti eg}
	\mathcal{G}(f)(x,v,t):=\frac{\rho(x,t)}{\sqrt{\det\left(2 \pi \mathcal{T}(x,t) \right)}}\exp \left({-\frac{(v-U(x,t))^{\top}(\mathcal{T}(x,t))^{-1}(v-U(x,t))}{2}}\right).
\end{align}
The macroscopic local density $\rho(x,t)$, bulk velocity $U(x,t)$, stress tensor $\Theta(x, t)$ and total energy $E(x, t)$ are defined as follows:
\begin{align*}
\begin{split}
\rho(x,t)&:= \int_{\mathbb{R}^3}f(x,v,t)dv,\cr
\rho(x,t) U(x,t) &:= \int_{\mathbb{R}^3}vf(x,v,t)dv,\cr
\rho(x,t) \Theta(x,t) &:= \int_{\mathbb{R}^3}\left(v-U(x,t)\right)\otimes\left(v-U(x,t)\right) f(x,v,t)dv,\cr
E(x,t) &:= \int_{\mathbb{R}^3}\frac{1}{2}|v-U(x,t)|^2f(x,v,t)dv.
\end{split}
\end{align*}
The temperature tensor $\mathcal{T}(x,t)$ is given by
\begin{align*}
\mathcal{T}(x,t)&=\nu \Theta(x,t) + (1-\nu)T(x,t)I,
\end{align*}
where $I$ is a $3\times 3$ identity matrix. Note that when $\nu=0$, ES-BGK model \eqref{A-1} reduces to BGK model \eqref{BGK}, and the Gaussian reduces to the local Maxwellian
\begin{align}\label{M}
	\mathcal{M}(f)(x,v,t):=\frac{\rho(x,t)}{\sqrt{2 \pi T(x,t)}^3}\exp \left({-\frac{|v-U(x,t)|^2}{2T(x,t)}}\right).
\end{align}
Similar to BTE and BGK model, the relaxation operator of ES-BGK model has the same five-dimensional collision invariants $1$, $v$, $|v|^2/2$:
\begin{align} \label{cancellation}
	 \int_{\mathbb{R}^3} \big(\mathcal{G}(f)- f\big) \begin{pmatrix}
	 	1\\v\\\frac{1}{2}|v|^2
	 \end{pmatrix} dv=0,
\end{align}
which implies the conservation of mass, momentum and energy.
%:
%\[
%\frac{d}{dt}  \int_{\mathbb{T}^d\times\mathbb{R}^3} f \phi(v) dxdv=0.
%\]
Moreover, $H$-theorem was also verified for ES-BGK model in \cite{ALPP} (See also \cite{Y})
\[
\frac{d}{dt}\int_{\mathbb{R}^3}f\ln{f}dv =  \int_{\mathbb{R}^3} \left(\mathcal{G}(f)-f\right) \ln{f} dv  \leq 0.
\]

The hydrodynamic models consistent with ES-BGK model can be derived from the Chapmann-Enskkog expansion $f=f^0 + \varepsilon f^1 + \varepsilon^2f^2 + \cdots$. The first order approximation gives Navier-Stokes equations:
\begin{align}\label{NS}
\begin{split}
		&\frac{\partial \rho}{\partial t} + \nabla_{x} \cdot (\rho u) = 0, \vspace*{0.2 cm}\\
	&\frac{\partial}{\partial t} (\rho u) + \nabla_{x} \cdot (\rho\, u \otimes u + pI)  = \varepsilon\, \nabla_{x} \cdot \left(\mu \sigma(u)\right) \,, \vspace*{0.2 cm}\\
	&\frac{\partial}{\partial t} E + \nabla_{x} \cdot \left( \left( E+p \right) u \right) \vspace*{0.2 cm}=
	\varepsilon\, \nabla_{x} \cdot \left(  \mu \sigma(u) u + \kappa \nabla_x T \right),
\end{split}
\end{align}	
where pressure and strain rate tensor are given by
$$p=\rho T,\quad \sigma(u)= \nabla_x u + (\nabla u)^T - \frac{2}{d_v} \nabla_x \cdot u I.$$ The viscosity $\mu$ and heat conductivity $\kappa$ are 
\[
\mu=\frac{1}{1-\nu}\frac{p}{\tau},\quad \kappa=\frac{d_v+2}{2}\frac{p}{\tau},
\]
and hence the Prandtl number is
\[
\text{Pr}=\frac{d_v+2}{2}\frac{\mu}{\kappa}=\frac{1}{1-\nu}.
\]
Thus, by taking $\nu=-\frac{1}{2}$, one can reproduce $\text{Pr}=\frac{2}{3}$. Note that when $\nu=0$ ES-BGK model reduces to BGK model, which further leads to $\text{Pr}=1$. Therefore, suitable choices of $\nu$ and $\tau$ result in the desired hydrodynamic limit at the Navier-Stokes level.

%%%%%%%%%%%%%%%%%%%%%%%%%%%%%%%%%%%%%%%%%
In this paper, we introduce a class of high-order conservative finite-difference semi-Lagrangian methods for the ES-BGK model. The idea of SL method is to integrate the solution along the associated characteristic curve. Although this nature make it necessary to
compute solutions on off-grid points, the method enables one to avoid CFL-type time step restriction from the convection term. 
Due to this benefit, the SL approach has been widely adopted in the context of numerical methods for Boltzmann-type equations \cite{BCR2,CBRY2,BCGR,CGQRY} or Vlasov-type models \cite{CMS,FSB,QC,SRB}, and for the development of SL methods, a key part is to choose suitable reconstruction.

In particular, when treating BTE or its BGK-type approximation models, two main features of solutions should be taken into account for reconstruction. First, solutions of such models could evolve to form a shock within a finite time. Thus, simple high order Lagrange interpolation may be inadequate because it may produce oscillations near discontinuities. The problem can be mitigated by adopting a WENO methodology. Second, conservative variables such as mass/momentum/energy should be preserved. However, nonlinear weights used for a WENO-type method may lead to loss of conservation. 

To address two issues altogether, a conservative reconstruction technique \cite{CBRY1} has been proposed in our recent paper \cite{CBRY1}, and adopted for the construction of high order conservative SL methods for BGK models for monoatomic \cite{CBRY2} and mixtures \cite{BCGR}, Vlasov Poisson system \cite{CBRY2}, and BTE \cite{BCR2}. The conservative reconstruction technique is based on the integration of the so-called basic reconstructions. A good candidate for the basic reconstruction is CWENO polynomials \cite{C,LPR} which guarantee the uniform accuracy within a cell. We confirmed that the reconstruction allows us to attain high order accuracy, maintain conservation under shifted summation, avoid spurious oscillations near discontinuities (see \cite{CBRY1}).

On top of previously mentioned strategy, we will employ a $L^2$-projection technique \cite{BCR1} to prevent loss of conservation that arises from continuous Gaussian function. More precisely, the problem occurs when the number of velocity girds for velocity discretization is insufficient to resolve the shape of Gaussian and reproduce its moments. This technique enable us to secure conservation even when the number of velocity grid is insufficient.

For time discretization, we will adopt $L$-stable diagonally implicit Runge-Kutta method or backward difference formula. Although a class of implicit semi-Lagrangian method for BGK model has been already proposed, it is non-trivial to extend the approach to ES-BGK model directly. This is because the Gaussian function involves non-conservative quantity such as temperature tensor, which needs to be dealt with implicitly. We will introduce how to compute the temperature tensor without using a Newton solver. We remark that a similar technique has been adopted in the development of Eulerian based methods. One is the first order scheme \cite{FJ} and the other is high order method \cite{H}. In the context of SL, we refer to \cite{RY}.

Finally, we will present numerical evidences to demonstrate the benefits of ES-BGK model and SL approach based on the comparison numerical solutions to ES-BGK model with the compressible Navier-Stokes equations and BTE.

The outline of this paper is as follows: In section \ref{sec numerical method}, we derive the semi-Lagrangian methods for the ES-BGK model. Next, in section \ref{sec property}, we study its mathematical properties. Then in Section \ref{sec numerical test}, we will present various numerical tests to demonstrate the performance of the proposed methods. Finally, the conclusion of this paper is presented.

%%%%%%%%%%%%%%%%%%%%%%%%%%%%%%%%%%%%%%%%%%%%%%%%%%%%%%%%%%%%%%%%%%%%%%%%%%%%%%%%%%%%%%%%%%%%%%%%%%%%%%%%%%%%%%%%%%%%%%%%%%%%%%%%%%%%%%%%%%%%%%%%%%%%%%%%%%%%%%%%%%%%%%%%%%%%%%%%%%%%%%%%%%%%%%%%%%%%%%%%%%%%%
%%%%%%%%%%%%%%%%%%%%%%%%%%%%%%%%%%%%%%%%%%%%%%%%%%%%%%%%%%%%%%%%%%%%%%%%%%%%%%%%%%%%%%%%%%%%%%%%%%%%%%%%%%%%%%%%%%%%%%%%%%%%%%%%%%%%%%%%%%%%%%%%%%%%%%%%%%%%%%%%%%%%%%%%%%%%%%%%%%%%%%%%%%%%%%%%%%%%%%%%%%%%%

\section{Derivation of the SL method}\label{sec numerical method}
	In this section, we describe how to derive semi-Lagrangian methods for ES-BGK model. For simplicity, we assume one-dimension in space and three-dimensions in velocity ($d_v=3$) for the derivation.
\subsection{Notation}
Let us consider computational domain $[0,T_f]\times [x_L,x_R] \times [v_{min},v_{max}]^{3}\in \mathbb{R}\times \mathbb{R}\times \mathbb{R}^3$. We discretize the time variable as
\begin{align*}
	t^n&=n\Delta t, \quad n=0,1,\dots,N_t,
\end{align*}
where $N_t\Delta t= T$. For spatial discretization, we use the uniform grid with $\Delta x=\frac{x_R-x_L}{N_x}$. For the periodic boundary conditions,
we adopt 
\begin{align*}
	x_i&=x_L + i \Delta x, \quad i=0,\cdots,N_x-1,
\end{align*}
while for the free-flow boundary conditions we set 
\begin{align*}
	x_i&=x_L + \left(i+\frac{1}{2}\right) \Delta x, \quad i=0,\cdots,N_x-1.
\end{align*}
To choose velocity grids, we discretize the truncated velocity domain $[v_{min},v_{max}]^{3}$ with the same size of mesh $\Delta v$ in each direction:
\begin{align*}
	v_j&=v_{min}(1,1,1)^T + (j_1 \Delta v,j_2 \Delta v,j_3 \Delta x)^T, \quad 0 \leq j_1,j_2,j_3\leq N_v,
\end{align*}
so that $N_v \Delta v=v_{max}-v_{min}$. For brevity, we set $v_{min}=-v_{max}$ in the rest of this paper. The numerical solution of $f$ at $(x_i,v_j,t^n)$ will be denoted by $f_{i,j}^n$. Similarly, we define $\mathcal{G}_{i,j}^n$ as
	\begin{align*}
		\mathcal{G}_{i,j}^n&=\frac{\rho_i^n}{\sqrt{\det\left(2 \pi \mathcal{T}_i^n \right)}}\exp \left({-\frac{(v_j-U_i^n)^{\top}(\mathcal{T}_i^n)^{-1}(v_j-U_i^n)}{2}}\right),
	\end{align*}
	where discrete macroscopic variables at $(x_i,t^n)$ are defined by
	\begin{align*} %\label{tensor n}
			\left(\rho_i^n,\rho_i^n U_i^n,E_i^n,\frac{3}{2}\rho_i^nT_i^n\right)&:= \sum_{j} f_{i,j}^n\left(1,v_j,\frac{1}{2}|v_j|^2, \frac{1}{2}|v_j-U_i^n|^2\right) (\Delta v)^{3},\\
			\rho_i^n\Theta_i^n &:= \sum_{j} f_{i,j}^{n}(v_j-U_i^{n}) \otimes (v_j-U_i^{n})(\Delta v)^{3},\\
			\Sigma_i^n&:= \sum_{j} f_{i,j}^{n}v_j \otimes v_j(\Delta v)^{3},\\
			\mathcal{T}_i^n&:=\nu  \Theta_i^n + (1-\nu)T_i^nI.
	\end{align*}

	\subsection{Derivation of the first order scheme} 
	Let us consider the characteristics of ES-BGK model \eqref{A-1}:
	\begin{align}\label{characteristic}
		\begin{split}
		\frac{df}{dt}&=\frac{\tau}{\varepsilon}\left(\mathcal{G}(f)-f\right),\cr
		\frac{dx}{dt}&=v,\\
		x(0)&=\tilde{x},\quad f(x,v,0)=f^0(x,v), \quad t\geq 0, \quad x\in \mathbb{R},\,v\in \mathbb{R}^3,
		\end{split}
	\end{align}
where $\tilde{x}$ denotes the characteristic foot and $f^0(x,v)$ is the given function. As in \cite{RY}, we use the collision invariants $1, v, v^2/2$ for the deriviation of our method. 
%
%The characteristic curves associated to the ES-BGK model \eqref{A-1} are the solutions of the following
%first order differential system over the time interval $[t^n, t^{n+1}]$:
%\begin{align*}
%	\frac{dX}{dt}=v_j, \quad X(t^{n+1};t^{n+1},x_i,v_j)=x_i,
%\end{align*}
%which immediately gives $X(t^{n};t^{n+1},x_i,v_j)=x_i-v_j \Delta t$. From this, 
%	\begin{align}\label{B-3}
%		\begin{split}
%		\frac{df}{ds}&=\frac{\tau}{\varepsilon}\left(\mathcal{G}(f)-f\right),\cr
%		\frac{dx}{ds}&=v_j,
%		\end{split}
%	\end{align}
%Here, one can easily have
%\[\displaystyle x(s)= x_i-v_j(t^{n+1}-s).\]
For treating the stiffness in $\tau/\varepsilon$, we begin by applying the implicit Euler method to the system \eqref{characteristic}:
\begin{align}\label{B-4}
	f_{i,j}^{n+1}-\tilde{f}_{i,j}^n = \frac{\tau_i^{n+1}\Delta t}{\varepsilon} \left(\mathcal{G}_{i,j}^{n+1}-f_{i,j}^{n+1}\right),
\end{align}
where $\tilde{f}_{i,j}^n$ is the approximation of $f$ on the characteristic foot $\tilde{x}_{i,j}=x_i-v_{j_1}\Delta t$ at time $t^n$. Multiplying both sides of \eqref{B-4} by collision invariants $\phi_{j}:=\left(1,v_j,\frac{1}{2}|v_j|^2\right)$ and taking a summation over $j$, we obtain
	\begin{align*}
		\sum_{j}\left(f_{i,j}^{n+1}-\tilde{f}_{i,j}^n\right) \phi_{j} (\Delta v)^3  = \sum_{j}\frac{\tau_i^{n+1}\Delta t}{\varepsilon} \left(\mathcal{G}(f_{i,j}^{n+1})-f_{i,j}^{n+1}\right) \phi_{j} (\Delta v)^3 .
	\end{align*}	
	Recalling the cancellation property \eqref{cancellation}, the right hand side becomes negligible if we secure sufficiently large velocity domain and take fine grids with a suitable quadrature rule. Assuming that the distribution function decays fast near the boundary, we obtain
    \begin{align}\label{rhouE}
\begin{split}
	    	\rho_i^{n+1} &= \sum_{j} \tilde{f}_{i,j}^n (\Delta v)^3 =: \tilde{\rho}_i^n, \cr
	U_i^{n+1} &= \frac{1}{\tilde{\rho}_i^n}\sum_{j} \tilde{f}_{i,j}^n v_j(\Delta v)^3 =: \tilde{U}_i^n,\cr
	E_i^{n+1} &= \sum_{j} \tilde{f}_{i,j}^{n} \frac{|v_j|^2}{2} (\Delta v)^3 =:\tilde{E}_i^n.
\end{split}
    \end{align}
From this, we can further approximate 
\begin{align}\label{TTT}
	\begin{split}
	T_i^{n+1}
%	&=\frac{2}{3}\frac{1}{{\rho}_i^{n+1}}\sum_{j} f_{i,j}^{n+1} \frac{|v_j-{U}_i^{n+1}|^2}{2}(\Delta v)^3 \cr
%	&=\frac{2}{3}\frac{1}{{\rho}_i^{n+1}}\left(E_i^{n+1}-\frac{\rho_i^{n+1}|U_i^{n+1}|^2}{2}\right) \cr
%	&=\frac{2}{3}\frac{1}{{\rho}_i^{n+1}}\left(E_i^{n+1}-\sum_{j} f_{i,j}^{n+1} v_j (\Delta v)^3 \cdot U_i^{n+1} +\sum_{j} f_{i,j}^{n+1}  (\Delta v)^3\frac{|U_i^{n+1}|^2}{2}\right) \cr
%	&=\frac{2}{3}\frac{1}{{\rho}_i^{n+1}}\left(\tilde{E}_i^{n}-\tilde{\rho}_i^n|\tilde{U}_i^{n}|^2 +\frac{\tilde{\rho}_i^n|\tilde{U}_i^{n}|^2}{2}\right) \cr
	&=  \frac{2}{3}\frac{1}{\tilde{\rho}_i^n}\sum_{j} \tilde{f}_{i,j}^{n} \frac{|v_j-\tilde{U}_i^n|^2}{2}(\Delta v)^3 =:\tilde{T}_i^n.
	\end{split}
\end{align}
In the case of implicit SL methods for BGK model in \cite{CBRY2,GRS}, the explicit approximation of $\rho_i^{n+1},u_i^{n+1},T_i^{n+1}$ is sufficient to determine the local Maxwellian, and hence no difficulty arises. However, for ES-BGK model, it remains to compute the temperature tensor $\mathcal{T}_i^{n+1}$. To avoid the use of any implicit solver, we introduce how to compute it explicitly. First, we multiply $\xi_{j}:=v_j \otimes v_j$ to (\ref{B-4}) to derive 	
\begin{align}\label{2.5}
	\begin{split}
		\Sigma_i^{n+1}&=\sum_{j} \left(\tilde{f}_{i,j}^{n} + \frac{\tau_i^{n+1} \Delta t}{\varepsilon} \mathcal{G}(f_{i,j}^{n+1})\right)\xi_{j}
		(\Delta v)^3 - \frac{\tau_i^{n+1} \Delta t}{\varepsilon}\Sigma_i^{n+1}\cr
		&= \Sigma_i^{*} + \frac{\tau_i^{n+1} \Delta t}{\varepsilon} \left(  \rho_i^{n+1} \left(\nu  \Theta_i^{n+1} + (1-\nu)T_i^{n+1}I\right)+ \rho_i^{n+1}u_i^{n+1}\otimes u_i^{n+1}\right) - \frac{\tau_i^{n+1} \Delta t}{\varepsilon}\Sigma_i^{n+1}\cr
		&= \Sigma_i^{*} + \frac{\tau_i^{n+1} \Delta t}{\varepsilon} \left( \nu\Sigma_i^{n+1}+   (1-\nu)\rho_i^{n+1}\left(T_i^{n+1}I + u_i^{n+1}\otimes u_i^{n+1}\right)\right) - \frac{\tau_i^{n+1} \Delta t}{\varepsilon}\Sigma_i^{n+1}\cr
		&= \Sigma_i^{*} + \frac{\tau_i^{n+1} \Delta t}{\varepsilon} \left(  (1-\nu)\rho_i^{n+1}\left(T_i^{n+1}I + u_i^{n+1}\otimes u_i^{n+1}\right)\right) - (1-\nu)\frac{\tau_i^{n+1} \Delta t}{\varepsilon}\Sigma_i^{n+1}
	\end{split}
\end{align}
where $\Sigma_i^{*}=\sum_{j} \left(\tilde{f}_{i,j}^{n} + \frac{\tau_i^{n+1} \Delta t}{\varepsilon} \mathcal{G}(f_{i,j}^{n+1})\right)\xi_{j}
(\Delta v)^3$.
Rearranging this, we get the explicit form of $\Sigma_i^{n+1}$:
\begin{align}\label{2.6}
	\Sigma_i^{n+1}= \frac{\varepsilon}{\varepsilon + (1-\nu)\tau_i^{n+1} \Delta t}\Sigma_i^{*}  + \frac{(1-\nu)\tau_i^{n+1} \Delta t}{\varepsilon + (1-\nu)\tau_i^{n+1} \Delta t}\rho_i^{n+1}\left(  T_i^{n+1}I + u_i^{n+1}\otimes u_i^{n+1}\right).
\end{align}
Now, we use this and the relation $\rho_i^{n+1}\Theta_i^{n+1}=\Sigma_i^{n+1} - \rho_i^{n+1} u_i^{n+1}\otimes u_i^{n+1}$ to approximate the temperature tensor $\mathcal{T}_i^{n+1}$ as follows:
\begin{align*}
	\mathcal{T}_i^{n+1}&= (1-\nu)T_i^{n+1}I + \nu \Theta_i^{n+1}\cr
	&= (1-\nu)T_i^{n+1}I + \frac{\nu}{\rho_i^{n+1}} \left[\frac{\varepsilon}{\varepsilon + (1-\nu)\tau_i^{n+1} \Delta t}\Sigma_i^{*}  + \frac{(1-\nu)\tau_i^{n+1} \Delta t}{\varepsilon + (1-\nu)\tau_i^{n+1} \Delta t}\rho_i^{n+1}\left(  T_i^{n+1}I + u_i^{n+1}\otimes u_i^{n+1}\right)
	\right]\cr
	&- \nu u_i^{n+1}\otimes u_i^{n+1}\cr
	&= (1-\nu)T_i^{n+1}I + \frac{\nu}{\rho_i^{n+1}} \left[\frac{\varepsilon}{\varepsilon + (1-\nu)\tau_i^{n+1} \Delta t}\Sigma_i^{*}  + \frac{(1-\nu)\tau_i^{n+1} \Delta t}{\varepsilon + (1-\nu)\tau_i^{n+1} \Delta t}\rho_i^{n+1} T_i^{n+1}I
\right]\cr	
&-\nu \left[ \frac{\varepsilon}{\varepsilon + (1-\nu)\tau_i^{n+1} \Delta t}\left(  u_i^{n+1}\otimes u_i^{n+1}\right)
\right].
\end{align*}
Then, we have a compact representation of $\mathcal{T}_i^{n+1}$:
\begin{align*}
	\mathcal{T}_i^{n+1}
	&= (1-\nu)T_i^{n+1}I + \nu \left[ \frac{(1-\nu)\tau_i^{n+1} \Delta t}{\varepsilon + (1-\nu)\tau_i^{n+1} \Delta t} T_i^{n+1} I
	\right]  + \nu \left[\frac{\varepsilon}{\varepsilon + (1-\nu)\tau_i^{n+1} \Delta t}\left(\frac{\Sigma_i^{*}}{\rho_i^{n+1}}- u_i^{n+1}\otimes u_i^{n+1} \right)
	\right]\cr
	&= (1-\nu)\left[ \frac{\varepsilon + \tau_i^{n+1} \Delta t }{\varepsilon + (1-\nu)\tau_i^{n+1} \Delta t} 
	\right]T_i^{n+1} I  +  \left[\frac{\varepsilon \nu}{\varepsilon + (1-\nu)\tau_i^{n+1} \Delta t}\left(\frac{\Sigma_i^{*}}{\rho_i^{n+1}}- u_i^{n+1}\otimes u_i^{n+1} \right)
	\right]\cr
%		&=\left[ \frac{ \varepsilon +  (1-\nu)\tau_i^{n+1} \Delta t -\varepsilon\nu }{\varepsilon + (1-\nu)\tau_i^{n+1} \Delta t} 
%	\right]T_i^{n+1} Id \cr
%	&+  \left[\frac{\varepsilon \nu}{\varepsilon + (1-\nu)\tau_i^{n+1} \Delta t}\left(\frac{\Sigma_i^{*}}{\rho_i^{n+1}}- u_i^{n+1}\otimes u_i^{n+1} \right)
%	\right]\cr
	&=(1-\nu_i^{n+1})T_i^{n+1} I + \nu_i^{n+1} \left(\frac{\Sigma_i^{*}}{\rho_i^{n+1}}- u_i^{n+1}\otimes u_i^{n+1} \right)\cr
	&=\tilde{\mathcal{T}}_i^n
\end{align*}
where 
\begin{align}\label{new nu}
	\nu_i^{n+1}=\frac{\varepsilon \nu}{\varepsilon + (1-\nu)\tau_i^{n+1} \Delta t}.
\end{align}
Note that $\tau_i^{n+1}$ depends on $\rho_i^{n+1}$ or $T_i^{n+1}$, thus it can also be computed explicitly. To sum up, the discrete Gaussian function $\mathcal{G}_{i,j}^{n+1}$ can be obtained by
\begin{align*}
	\mathcal{G}_{i,j}^{n+1}=\tilde{\mathcal{G}}_{i,j}^{n}:&= \frac{\tilde{\rho}_i^n}{\sqrt{\det\left(2 \pi \tilde{\mathcal{T}}_i^n \right)}} \exp \left(-\frac{(v_j-\tilde{U}_i^n)^{\top} (\tilde{\mathcal{T}}_i^n )^{-1}(v_j-\tilde{U}_i^n)}{2}\right).
\end{align*}
%Note that $\Lambda_\delta:= \sum_k e^{-I_k^{\frac{2}{\delta}}} $ is precomputable and all discrete macroscopic quantities are explicitly computed by $\tilde{f}_{i,j}^n$ and formula \eqref{rhouE}, \eqref{TTT}, \eqref{B-5}, \eqref{Ttheta}.
Finally, we our scheme reads
\begin{align}\label{numsol}
f_{i,j}^{n+1} =\frac{\varepsilon\tilde{f}_{i,j}^n +   \tau_i^{n+1} \Delta t \tilde{\mathcal{G}}_{i,j}^{n}}{\varepsilon + \tau_i^{n+1} \Delta t}.
\end{align}
Note that it is sufficient to employ the linear interpolation to construct first order scheme.
\begin{remark}
	The way of approximating the temperature tensor $\mathcal{T}_i^{n+1}$ in an explicit way is similar to the implicit explicit method adopted in the Eulerian framework \cite{FJ}. In particular, the form of $\Sigma_i^{n+1}$ in our method is consistent with the Eulerian method. We also note that our approximation of temperature tensor $\mathcal{T}_i^{n+1}$ is equivalent to the method in \cite{RY}.
\end{remark}

%%%%%%%%%%%%%%%%%%%%%%%%%%%%%%%%%%%%%%%%%%%%%%%%%%%%%%%%%%%%%%%%%%%%%%%%%%%%%%%%%%%%%%%%%%%%%%%%%%%%%%%%%%%%%%%%%%%%%%%%%%%%%%%%%%%%%%%%%%%%%%%%%%%%%%%%%%%%%%%%%%%%%%%%%%%%%%%%%%%%%%%%%%%%%%%%%%%%%%%%%%%%

	\subsection{Derivation of the high order scheme} 
	In this section, we introduce two class of high order SL methods for ES-BGK model. One is $L$-stable diagonally implicit Runge-Kutta method (DIRK) and the other one is backward differene formula (BDF). Before proceeding, for the brevity of description, we define $$Q(x,v,t):=\tau(x,v,t)\left(\mathcal{G}(x,v,t)-f(x,v,t)\right).$$ 

	\subsubsection{DIRK methods}
We begin by applying $s$-stage DIRK methods to \eqref{characteristic}:
\begin{align}\label{rk 1st}
	f_{i,j}^{(k)} &= \tilde{f}_{i,j}^{(k,0)}+\sum_{\ell=1}^{k-1}a_{k\ell}\frac{\Delta t}{\varepsilon} Q_{i,j}^{(k,\ell)} +a_{kk} \frac{\tau_i^{(k)}\Delta t}{\varepsilon} \left(\mathcal{G}_{i,j}^{(k)}-f_{i,j}^{(k)}\right),\quad k=1,...,s,
\end{align}	
where 
$\tilde{f}_{i,j}^{(k,0)}\approx  f(x_i-v_{j_1}c_k\Delta t,v_j,t^n)$ and
$Q_{i,j}^{(k,\ell)}\approx Q(x_i-v_{j_1}(c_k-c_\ell)\Delta t,v_j,t^n+(c_k-c_\ell)\Delta t)$. Letting
\begin{align}\label{RK tilde}
	\tilde{f}_{i,j}^{(k)}:=\tilde{f}_{i,j}^{(k,0)}+\sum_{\ell=1}^{k-1}a_{k\ell}\frac{\Delta t}{\varepsilon} Q_{i,j}^{(k,\ell)},
\end{align}
we can rewrite \eqref{rk 1st} as
\begin{align}
	f_{i,j}^{(k)} = \tilde{f}_{i,j}^{(k)}+a_{kk} \frac{\tau_i^{(k)}\Delta t}{\varepsilon} \left(\mathcal{G}_{i,j}^{(k)}-f_{i,j}^{(k)}\right).
\end{align}	
Note that this form is similar to the first order method in \eqref{B-4}. Using the same argument as in the first order method, we can explicitly compute the conservative moments 
\begin{align}
	\begin{split}
		\rho_i^{(k)} &= \sum_{j} \tilde{f}_{i,j}^{(k)} (\Delta v)^3 =: \tilde{\rho}_i^{(k)}, \cr
		U_i^{(k)} &= \frac{1}{\tilde{\rho}_i^n}\sum_{j} \tilde{f}_{i,j}^{(k)} v_j(\Delta v)^3 =: \tilde{U}_i^{(k)},\cr
		E_i^{(k)} &= \sum_{j} \tilde{f}_{i,j}^{(k)} \frac{|v_j|^2}{2} (\Delta v)^3 =:\tilde{E}_i^{(k)},
	\end{split}
\end{align}
and the temperature
\begin{align}
	\begin{split}
		T_i^{(k)}&=\frac{2}{3}\frac{1}{{\rho}_i^{(k)}}\sum_{j} f_{i,j}^{(k)} \frac{|v_j-{U}_i^{(k)}|^2}{2}(\Delta v)^3 \cr
		&=  \frac{2}{3}\frac{1}{\tilde{\rho}_i^{(k)}}\sum_{j} \tilde{f}_{i,j}^{(k)} \frac{|v_j-\tilde{U}_i^n|^2}{2}(\Delta v)^3\cr
		&=:\tilde{T}_i^{(k)}.
	\end{split}
\end{align}
For the stress tensor $\Theta_i^{n+1}$, we introduce	
\begin{align}
	\tilde{\Sigma}_i^{(k)}&=\sum_{j} \tilde{f}_{i,j}^{(k)}\xi_{j}
	(\Delta v)^3
\end{align}
which gives
\begin{align}
	\Sigma_i^{(k)}= \frac{\varepsilon}{\varepsilon + (1-\nu)\tau_i^{(k)} a_{kk}\Delta t}\tilde{\Sigma}_i^{(k)}  + \frac{(1-\nu)\tau_i^{(k)} a_{kk}\Delta t}{\varepsilon + (1-\nu)\tau_i^{(k)} a_{kk}\Delta t}\rho_i^{{(k)}}\left(  T_i^{(k)}I + u_i^{(k)}\otimes u_i^{(k)}\right).
\end{align}
This, combined with the relation $\rho_i^{(k)}\Theta_i^{(k)}=\Sigma_i^{(k)} - \rho_i^{(k)} u_i^{(k)}\otimes u_i^{(k)}$ enable us to approximate $\mathcal{T}_i^{(k)}$ as 
\begin{align*}
	\mathcal{T}_i^{(k)}
	&=\left[ \frac{ \varepsilon +  (1-\nu)\tau_i^{(k)} a_{kk}\Delta t -\varepsilon\nu }{\varepsilon + (1-\nu)\tau_i^{(k)} a_{kk}\Delta t} 
	\right]T_i^{(k)} I +  \left[\frac{\varepsilon \nu}{\varepsilon + (1-\nu)\tau_i^{(k)} a_{kk}\Delta t}\left(\frac{\tilde{\Sigma}_i^{(k)}}{\rho_i^{(k)}}- u_i^{(k)}\otimes u_i^{(k)} \right)
	\right]\cr
	&=(1-\nu_i^{(k)})T_i^{(k)} I + \nu_i^{(k)} \left(\frac{\tilde{\Sigma}_i^{(k)}}{\rho_i^{(k)}}- u_i^{(k)}\otimes u_i^{(k)} \right)\cr
	&=\tilde{\mathcal{T}}_i^{(k)},
\end{align*}
where
\begin{align}\label{new nu high}
	\nu_i^{(k)}=\frac{\varepsilon \nu}{\varepsilon + (1-\nu)\tau_i^{(k)} a_{kk}\Delta t}.
\end{align}
To sum up, we can explicitly compute $	\mathcal{G}_i^{(k)}$ with
\begin{align*}
	\mathcal{G}_i^{(k)}=\tilde{\mathcal{G}}_{i,j}^{(k)}:&= \frac{\tilde{\rho}_i^{(k)}}{\sqrt{\det\left(2 \pi \tilde{\mathcal{T}}_i^{(k)} \right)}} \exp \left(-\frac{(v_j-\tilde{U}_i^{(k)})^{\top} (\tilde{\mathcal{T}}_i^{(k)} )^{-1}(v_j-\tilde{U}_i^{(k)})}{2}\right).
\end{align*}
%Note that $\Lambda_\delta:= \sum_k e^{-I_k^{\frac{2}{\delta}}} $ is precomputable and all discrete macroscopic quantities are explicitly computed by $\tilde{f}_{i,j}^n$ and formula \eqref{rhouE}, \eqref{TTT}, \eqref{B-5}, \eqref{Ttheta}.
Then, the $k$-stage value of $f_{i,j}^{(k)}$ is given by 
\begin{align}\label{numSol RK}
	f_{i,j}^{(k)} =\frac{\varepsilon\tilde{f}_{i,j}^{(k)} +   \tau_i^{(k)} a_{kk} \Delta t\mathcal{G}_i^{(k)}}{\varepsilon + \tau_i^{(k)} a_{kk}\Delta t}.
\end{align}
After updating solution from $k=1$ to $s$, we finally set
$f_{i,j}^{n+1}=f_{i,j}^{(s)}.$
This follows from the stiffly accurate (SA) property of the DIRK methods \cite{HG}.

%%%%%%%%%%%%%%%%%%%%%%%%%%%%%%%%%%%%%%%%%%%%%%%%%%%%%%%%%%%%%%%%%%%%%%%%%%%%%%%%%%%%%%%%%%%%%%%%%%%%%%%%%%%%%%%%%%%%%%%%%%%%%%%%%%%%%%%%%%%%%%%%%%%%%%%%%%%%%%%%%%%%%%%%%%%%%%%%%%%%%%%%%%%%%%%%%%%%%%%%%%%%

\subsubsection{BDF methods}
Applying $s$-step BDF methods to \eqref{characteristic}, we get
\begin{align}\label{bdf form}
	f_{i,j}^{n+1} &= \sum_{k=1}^s \alpha_k f_{i,j}^{n,k} +\beta_{s} \frac{\tau_i^{n+1}\Delta t}{\varepsilon} \left(\mathcal{G}_{i,j}^{n+1}-f_{i,j}^{n+1}\right),
\end{align}	
where $f_{i,j}^{n,k}\approx f(x_i-kv_{j_1}\Delta t,v_j,t^{n+1-k})$. This can be rewritten as
\begin{align}
	f_{i,j}^{n+1} &= \tilde{f}_{i,j}^{n} +\beta_{s} \frac{\tau_i^{n+1}\Delta t}{\varepsilon} \left(\mathcal{G}_{i,j}^{n+1}-f_{i,j}^{n+1}\right),
\end{align}
where $\tilde{f}_{i,j}^{n}=\sum_{k=1}^s \alpha_k f_{i,j}^{n,k}$. Hereafter, as in the DIRK based method, the same argument can be applied for the derivation of BDF based methods. Discrete macroscopic variables can be approximated as follows:
\begin{align}
	\begin{split}
		\rho_i^{n+1} &= \sum_{j} \tilde{f}_{i,j}^{n} (\Delta v)^3 =: \tilde{\rho}_i^{n}, \cr
		U_i^{n+1} &= \frac{1}{\rho_i^{n+1}}\sum_{j} \tilde{f}_{i,j}^{n} v_j(\Delta v)^3 =: \tilde{U}_i^{n},\cr
		E_i^{n+1} &= \sum_{j} \tilde{f}_{i,j}^{n} \frac{|v_j|^2}{2} (\Delta v)^3 =:\tilde{E}_i^{n},\cr
		T_i^{n+1}&=\frac{2}{3}\frac{1}{{\rho}_i^{n+1}}\sum_{j} \tilde{f}_{i,j}^{n} \frac{|v_j-{U}_i^{n+1}|^2}{2}(\Delta v)^3 \cr
		&=  \frac{2}{3}\frac{1}{\tilde{\rho}_i^{n}}\sum_{j} \tilde{f}_{i,j}^{n} \frac{|v_j-\tilde{U}_i^n|^2}{2}(\Delta v)^3 ,\cr
		&=:\tilde{T}_i^{n}.
	\end{split}
\end{align}
Denoting
\begin{align}
	\tilde{\Sigma}_i^{n}&=\sum_{j} \tilde{f}_{i,j}^{n}\xi_{j}
	(\Delta v)^3,
\end{align}
we compute
\begin{align}
	\Sigma_i^{n+1}= \frac{\varepsilon}{\varepsilon + (1-\nu)\tau_i^{n+1} \beta_s\Delta t}\tilde{\Sigma}_i^{n}  + \frac{(1-\nu)\tau_i^{n+1} \beta_s\Delta t}{\varepsilon + (1-\nu)\tau_i^{n+1} \beta_s\Delta t}\rho_i^{n+1}\left(  T_i^{n+1}I + u_i^{n+1}\otimes u_i^{n+1}\right).
\end{align}
Then, we combine this with $\rho_i^{n+1}\Theta_i^{n+1}=\Sigma_i^{n+1} - \rho_i^{n+1} u_i^{n+1}\otimes u_i^{n+1}$ to obtain
\begin{align*}
	\mathcal{T}_i^{n+1}
	&=\left[ \frac{ \varepsilon +  (1-\nu)\tau_i^{n+1} \beta_s\Delta t -\varepsilon\nu }{\varepsilon + (1-\nu)\tau_i^{n+1} \beta_s\Delta t} 
	\right]T_i^{n+1} I +  \left[\frac{\varepsilon \nu}{\varepsilon + (1-\nu)\tau_i^{n+1} \beta_s\Delta t}\left(\frac{\tilde{\Sigma}_i^{n}}{\rho_i^{n+1}}- u_i^{n+1}\otimes u_i^{n+1} \right)
	\right]\cr
	&=(1-\nu_i^{n+1})T_i^{n+1} I + \nu_i^{n+1} \left(\frac{\tilde{\Sigma}_i^{n}}{\rho_i^{n+1}}- u_i^{n+1}\otimes u_i^{n+1} \right)\cr
	&=\tilde{\mathcal{T}}_i^{n},
\end{align*}
where 
\[
\nu_i^{n+1}=\frac{\varepsilon \nu}{\varepsilon + (1-\nu)\tau_i^{n+1} \beta_s\Delta t}.
\]
To sum up, we derive
\begin{align*}
	\mathcal{G}_{i,j}^{n+1}=\tilde{\mathcal{G}}_{i,j}^{n}:&= \frac{\tilde{\rho}_i^{n}}{\sqrt{\det\left(2 \pi \tilde{\mathcal{T}}_i^{n} \right)}} \exp \left(-\frac{(v_j-\tilde{U}_i^{n})^{\top} (\tilde{\mathcal{T}}_i^{n} )^{-1}(v_j-\tilde{U}_i^{n})}{2}\right).
\end{align*}
The last step is to update solution with
\begin{align}\label{numSol bdf}
	f_{i,j}^{n+1} =\frac{\varepsilon\tilde{f}_{i,j}^{n} +   \tau_i^{n+1} \beta_s \Delta t\tilde{\mathcal{G}}_{i,j}^{n}}{\varepsilon + \tau_i^{n+1} \beta_s\Delta t}.
\end{align}

\subsection{Weighted $L^2$-minimization for moment correction}\label{sec weight}
In this section, we review the weighted $L^2$-minimization technique \cite{BCR1} which will be applied to the SL method for ES-BGK model. To describe the detail, let us consider a $\mathcal{G} \in \mathbb{R}^{{N_v}^{d_v}}$ and its macroscopic variables $\mathcal{U}:=(\rho,\rho U, E)^{\top} \in \mathbb{R}^{d_v+2}$. When the number of $N_v$ is not sufficient, it is difficult to reproduce $\mathcal{U}$ from $\mathcal{G}$. To treat this problem, we look for a vector $g\in \mathbb{R}^{{N_v}^{d_v}}$ that is close to $\mathcal{G}$ and satisfies two constraints, (1) the vector $g$ should reproduce the reference moments $\mathcal{U}$ (2) the vector $g$ should be very small near boundaries of the velocity domain. To achieve these goals, we solve a constrained weighted $L^2$-minimization problem:
%By solving the problem, we slightly modify the initial guess of the distribution function, i.e.,
%If we reconstruct the distribution function $f$ to match with the reference mass only, the other quantities can not be reproduced with $f$. 
%That is,
%given an initial guess $f \in \mathbb{R}^{{N_v}^{d_v}}$ we look for a solution $g\in \mathbb{R}^{{N_v}^{d_v}}$ of the following $L^2$-minimization problem:
\begin{align}\label{min prb}
	\min_g \bigg\|\mathcal{G}\circ \frac{1}{\omega} - g\bigg\|_2^2 \quad\text{ s.t }\quad Cg=\mathcal{U}
\end{align}
where $\circ$ denotes the component-wise multiplication and
$$\mathcal{G} \equiv (\mathcal{G}_1,\mathcal{G}_2,\cdots,\mathcal{G}_{{N_v}^{d_v}})^\top\in \mathbb{R}^{{N_v}^{d_v}}, \quad g =(g_1,g_2,\cdots,g_{{N_v}^{d_v}})^\top\in \mathbb{R}^{{N_v}^{d_v}},$$
$$\omega =(\omega_1,\omega_2,\cdots,\omega_{{N_v}^{d_v}})^\top\in \mathbb{R}^{{N_v}^{d_v}},$$
$$\displaystyle C:=\begin{pmatrix}
	\omega_j(\Delta v)^{d_v}\\  \omega_j v_j (\Delta v)^{d_v} \\ \frac{\omega_j |v_j|^2}{2} (\Delta v)^{d_v}
\end{pmatrix}\in \mathbb{R}^{(d_v+2) \times {N_v}^{d_v}}, \quad \mathcal{U}= (\rho,\rho U, E)^{\top}\in \mathbb{R}^{(d_v+2) \times 1}.$$
As in \cite{BCR1}, the solution can be obtained by the method of Lagrange multiplier. We first denote as $\mathcal{L}(g,\lambda)$ the Lagrangian: 
$$\mathcal{L}(g,\lambda)= \left\|\mathcal{G}\circ \frac{1}{\omega}- g\right\|_2^2 + \lambda^{\top}\left(C g-\mathcal{U}\right),$$
where $\lambda$ are the corresponding Lagrangian variables. Finding the gradients of $\mathcal{L}$ with respect to $g$ and $\lambda$ to be zero at the same time:
\begin{align*}
	\begin{array}{rlrl}
		\nabla_g \mathcal{L} =0 &\Leftrightarrow& g&= \mathcal{G}\circ \frac{1}{\omega} + \frac{1}{2}C^{\top} \lambda\cr
		\nabla_\lambda \mathcal{L} = 0 &\Leftrightarrow& C g&=\mathcal{U}
	\end{array},
\end{align*}
we can find the stationary points $\lambda$:
\begin{align*}
	%		Cg= C\mathcal{M}\circ \frac{1}{h}+ \frac{1}{2}CC^{\top} \lambda &\Leftrightarrow \mathcal{U}= C\mathcal{M}\circ \frac{1}{h}+ \frac{1}{2}CC^{\top} \lambda \cr
	%		&\Leftrightarrow  
	\lambda= 2(CC^{\top})^{-1}\left(\mathcal{U}-C\left(\mathcal{G}\circ \frac{1}{\omega}\right)\right).
\end{align*}
Since the matrix $CC^{\top}$ is symmetric and positive definite, it is invertible. Consequently,
\begin{align*}%\label{first dm}
	g\circ \omega= \mathcal{G}+C^{\top}(CC^{\top})^{-1}\left(\mathcal{U}-C\left(\mathcal{G}\circ \frac{1}{\omega}\right)\right)\circ \omega.
\end{align*}
\begin{remark}\label{rem delta}
	In practice, considering machine precision, the norm of $Cg-U$ should be considered a very small value. In the rest of this paper, we will assume that the solution to \eqref{min prb} satisfies
	$$\|Cg-U\|_\infty<\delta,$$
	where $\delta$ denotes the machine precision.
\end{remark}
%
%  
%As in \cite{BCR1}, we may take into account the weight for the optimization. 

%%%%%%%%%%%%%%%%%%%%%%%%%%%%%%%%%%%%%%%%%%%%%%%%%%%%%%%%%%%%%%%%%%%%%%%%%%%%%%%%%%%%%%%%%%%%%%%%%%%%%%%%%%%%%%%%%%%%%%%%%%%%%%%%%%%%%%%%%%%%%%%%%%%%%%%%%%%%%%%%%%%%%%%%%%%%%%%%%%%%%%%%%%%%%%%%%%%%%%%%%%%%%%%%%%%%%%%%%%%%%%%%%%%%%%%%%%%%%%%%%%%%%%%%%%%%%%%%%%%%%%%%%%%%%%%%%

\section{properties of the numerical method}\label{sec property}

In this section, we study the several properties of the proposed SL methods.

\begin{proposition}
%	The proposed first order methods satisfy
	%	\begin{enumerate}
		%		\item 
		Given $\varepsilon>0$ and $t^n > 0$, the first order methods satisfies
		$$0\leq f_{i,j}^{n+1}\leq \max\{\|f^n\|_\infty ,\|\mathcal{G}^{n+1}\|_\infty\}.$$
		%		\item For all $\Delta t> 0$ and $f_0$, the distribution function $f^n$ converges to $\mathcal{M}^n$, that is,
		%		$\lim\limits_{\varepsilon \to 0} f^n = \mathcal{M}^n$
		%		and the scheme gives a first order approximation in time of the compressible Euler system.
		%	\end{enumerate}
	
\end{proposition}

\begin{proof}
	%		\noindent \textbf{Proof of (1)} 
	In the first order scheme \eqref{numsol}, the numerical solution $f_{i,j}^{n+1}$ is given by the convex combination of $\tilde{f}_{i,j}^n$ and $\tilde{\mathcal{G}}_{i,j}^{n}=\mathcal{G}_{i,j}^{n+1}$. Furthermore, the value of $\tilde{f}_{i,j}^n$ is obtained by a linear interpolation of $f_{i^*,j}^n$ and $f_{i^*+1,j}^n$ where $i^*$ is the index such that $x_{i^*-1}< x_i-v_j\Delta t\leq x_{i^*}$. This implies the desired result.
	%		\noindent \textbf{Proof of (2)}\\
	%		Given a distribution function $f^n$, the asymptotic limit $\varepsilon \to 0$ gives
	%		\begin{align*}
		%			\lim\limits_{\varepsilon \to 0} \Sigma_i^{n+1}&=\lim\limits_{\varepsilon \to 0}\left[ \frac{\varepsilon}{\varepsilon + (1-\nu)\tau_i^{n+1} \Delta t}\Sigma_i^{*}  + \frac{(1-\nu)\tau_i^{n+1} \Delta t}{\varepsilon + (1-\nu)\tau_i^{n+1} \Delta t}\rho_i^{n+1}\left(  T_i^{n+1}Id + u_i^{n+1}\otimes u_i^{n+1}\right)\right]\cr
		%			&=\rho_i^{n+1}\left(  T_i^{n+1}Id + u_i^{n+1}\otimes u_i^{n+1}\right)
		%		\end{align*}
	%		This further implies that  
	%		$$ \lim\limits_{\varepsilon \to 0}\rho_i^{n+1}\Theta_i^{n+1} = \lim\limits_{\varepsilon \to 0}\left[\Sigma_i^{n+1} - \rho_i^{n+1} u_i^{n+1}\otimes u_i^{n+1}\right] = \rho_i^{n+1} T_i^{n+1}Id$$ 
	%		and hence
	%		\begin{align*}
		%			\lim\limits_{\varepsilon \to 0}\mathcal{T}_i^{n+1}
		%			&=T_i^{n+1} Id
		%		\end{align*}
	%		To sum up, in the limit $\varepsilon \to 0$, the distribution function converges to isotropic Maxwellian 
	%		\[
	%		\mathcal{M}_{i,j}^{n+1}=
	%		\frac{\rho_i^{n+1}}{(2 \pi T_i^{n+1})^{\frac{d_v}{2}}} \exp \left(-\frac{|v_j-U_i^{n+1}|^2}{2T_i^{n+1}}\right).
	%		\]

\end{proof}

In the following proposition, we show that our first order SL scheme is consistent to compressible Navier-Stokes equations for small Knudsen number.
\begin{proposition}
	Let $f^n(x,v)$ be a solution with time discretization. Assume there exist positive constants $C_1, C_2$ independent of $n$ such that for each $n\geq 0$
	\begin{align}\label{con1}
		\sum_{0\leq |\alpha| \leq 2} \sup_{x,v} {|\partial_x^\alpha f^n(x,v)|(1+ |v|^8)} < C_1,
	\end{align}
	and
	\begin{align}\label{con2}
	\sum_{0\leq |\alpha| \leq 2} \sup_{x,v} {|\partial_x^\alpha \mathcal{M}^n(x,v)|(1+ |v|^8)} < C_2,
\end{align}
where $\mathcal{M}$ denotes the local Maxwellian defined in \eqref{M}. Then, the first order (in time) SL scheme \eqref{numsol} for ES-BGK model asymptotically becomes a consistent approximation of the compressible Navier-Stokes system \eqref{NS} for small Knudsen number $\varepsilon$.
\end{proposition}

%
%\begin{proposition}\label{propo consistency}
%	Let $f^n(x,v)$ be a solution with time discretization. Assume there exist positive constants $C_1, C_2$ independent of $n$ such that
%	\begin{align}\label{con1}
%		\sum_{0\leq |\alpha| \leq 2} \sup_{x,v} {|\partial_x^\alpha f^n(x,v)|(1+ |v|^8)} < C_1,
%	\end{align}
%	and
%	\begin{align}\label{con2}
%		\int_{\mathbb{R}} {f}^{n}(x,v)dv>C_2>0.
%	\end{align}
%	Then, the first order SL scheme \eqref{numSol} is consistent with the first order approximations in time of the compressible Navier-Stokes system.
%\end{proposition}

Consider the characteristic formation of \eqref{A-1}:
	\begin{align*}
		\begin{split}
			\frac{df}{ds}&=\frac{\tau}{\varepsilon}\left(\mathcal{G}(f)-f\right),\quad f(x(t),v(t),t)=f(x,v,t),\cr
			\frac{dx}{ds}&=v, \quad x(t^{n+1})=x,\quad 	\frac{dv}{ds}=0,\quad v(t^{n+1})=v.
		\end{split}
	\end{align*}
	Applying the implicit Euler method to this, we obtain
	\begin{align}\label{ap 1st order scheme}
		\frac{f^{n+1} - \tilde{f}^{n}}{\Delta t} &=   \frac{\tau^{n+1}}{\varepsilon}\left(\mathcal{G}({f}^{n+1})-f^{n+1}\right),
	\end{align}
	where $\tilde{f}^n:=f^n(x-v\Delta t,v)$. Now, we take moments of \eqref{ap 1st order scheme} with respect to $\phi(v)=1,v,|v|^2/2$ to obtain
	\begin{align}\label{ap int}
		\int_{\mathbb{R}} \phi(v) {f}^{n+1}dv
		= \int_{\mathbb{R}} \phi(v) \tilde{f}^n dv.
	\end{align}
By Taylor's theorem with remainder in the integral form, we get
	\begin{align}\label{R}
		\begin{split}
			\int_{\mathbb{R}} \phi(v) f^{n+1}dv
			&= \int_{\mathbb{R}} \phi(v) f^n(x-v\Delta t,v) dv\cr
			&= \int_{\mathbb{R}} \phi(v) \bigg[f^n(x,v) - v\Delta t \nabla_x f^n(x,v) \cr
			&\hspace{1.6cm}+2! \sum_{|\alpha|=2}\frac{(-v\Delta t)^\alpha}{\alpha !}  \int_0^{1} (1-u) \partial_{x}^\alpha f^n(x + u (-v\Delta t),v) du\bigg] dv\cr
			&=: \int_{\mathbb{R}} \phi(v) \left[f^n(x,v) - v\Delta t \partial_x f^{n}(x,v) \right] dv + R.
		\end{split}
	\end{align}
	From our assumption \eqref{con1} and $|\phi(v)|<1+|v|^2$, the remainder term $R$ is estimated as
	\begin{align}\label{R1}
		\begin{split}
			\left\|R\right\|_\infty &\leq\int_{\mathbb{R}^d} (1+|v|^2) \left|2! \sum_{|\alpha|=2}\frac{(-v\Delta t)^\alpha}{\alpha !}  \int_0^{1} (1-u) \partial_{x}^\alpha f^n(x + u (-v\Delta t),v) du\right|dv\cr
			&\leq \sup_{x,v}\left|\partial_{x}^2f^n(x,v)(1+|v|^2)^4\right|(\Delta t)^2 \int_{\mathbb{R}^d} \frac{6|v|^2}{(1+|v|^2)^3}\left|\int_0^1 (1-u) dy\right|dv  \cr
			&\leq \sup_{x,v}\left|\partial_{x}^2f^n(x,v)(1+|v|^2)^4\right| (\Delta t)^2 \int_{\mathbb{R}^d} \frac{3|v|^2}{(1+|v|^2)^3} dv  \cr
			&=  \mathcal{O}({\Delta t^2}).
		\end{split}
	\end{align}
We now expand $f$ as
\[
f^n= \mathcal{M}^n + \varepsilon f_1^n
\]
where $f_1^n$ satisfies the so-called compatibility conditions:
\[
\int_{\mathbb{R}^3}f_1^n (1,v,|v|^2)dv=0.
\]
The zeroth order approximation $f^n=\mathcal{M}^n$, together with \eqref{R} and \eqref{R1}, gives  
	\begin{align}\label{Euler}
		\begin{split}
			&\frac{ \rho^{n+1}-\rho^{n}}{\Delta t} + \nabla \cdot(\rho^n u^n) = \mathcal{O}({\Delta t}),\cr
			&\frac{ \rho^{n+1} u^{n+1}-\rho^{n} u^{n}}{\Delta t} + \nabla \cdot(\rho^{n} u^{n} \otimes u^n + p^{n} I) =  \mathcal{O}({\Delta t}),\cr
				&\frac{E^{n+1}- E^{n}}{\Delta t} + \nabla \cdot((E^n + p^n) u^n) =  \mathcal{O}({\Delta t}).
		\end{split}
	\end{align}
		With the first order approximation $f^n=\mathcal{M}^n +\varepsilon f_1^n$, we get
	\begin{align}\label{Theta}
		\Theta^n= T^n I + \varepsilon \Theta_1^n 
	\end{align}
	where
	\[
	\rho^n\Theta_1^n = \int_{\mathbb{R}^3} f_1^n (v-u^n)\otimes (v-u^n) dv,\quad  tr(\Theta_1^n)=0,
	\]
	\[
	Q_1^n= \int_{\mathbb{R}^3} \frac{|v-u^n|^2}{2}(v-u^n)f_1^n dv.
	\]
	Inserting $f^n= \mathcal{M}^n +\varepsilon f_1^n$ into \eqref{R}, we obtain
	\begin{align}\label{NSE BGK}
		\begin{split}
			&\frac{ \rho^{n+1}-\rho^{n}}{\Delta t} + \nabla \cdot(\rho^n u^n)  =\mathcal{O}({\Delta t}),\cr
			&\frac{ (\rho^{n+1} u^{n+1})-(\rho^{n} u^{n})}{\Delta t} + \nabla \cdot(\rho^{n} u^{n} \otimes u + p^{n} I) = -\varepsilon \nabla\cdot(\rho^n\Theta_1^n) + \mathcal{O}({\Delta t}),\cr
			&\frac{E^{n+1}- E^{n}}{\Delta t} + \nabla \cdot((E^n + p^n) u^n) = -\varepsilon \nabla\cdot(Q_1^n + \rho^n\Theta_1^nu^n)+ \mathcal{O}({\Delta t}).
		\end{split}
	\end{align}
	For the application of the Chapman-Enskog method, we expand the Gaussian $\mathcal{G}^n$ with respect to $\varepsilon$:
	$$\mathcal{G}^n = \mathcal{M}^n + \varepsilon g_1^n.$$
	The next step is to insert these expansions into the ES-BGK model
	and use the compatibility conditions, which yields for $n \geq 1$:
	\begin{align}
		\frac{\mathcal{M}^n + \varepsilon f_1^n - \left(\tilde{\mathcal{M}}^{n-1} + \varepsilon \tilde{f}_1^{n-1}\right)}{\Delta t} &=   \frac{\tau^{n}}{\varepsilon}\left(\mathcal{M}^n + \varepsilon g_1^n-\left(\mathcal{M}^n + \varepsilon f_1^n \right)\right),
	\end{align}
which can be rearranged as follows
	\[
	\frac{\mathcal{M}^{n} - \tilde{\mathcal{M}}^{n-1}}{\Delta t} =\tau^n(g_1^n-f_1^n) - \varepsilon \left(\frac{f_1^{n}- \tilde{f}_1^{n-1}}{\Delta t}\right).
	\]
	Since the assumption \eqref{con2} implies 
\begin{align*}
		\tilde{\mathcal{M}}^{n-1}&=\mathcal{M}(x-v\Delta t,v,t^{n-1})=\mathcal{M}^{n-1}(x,v) - \Delta tv \cdot  \nabla\mathcal{M}^{n-1}(x,v) + \mathcal{O}((\Delta t)^2),
\end{align*}	
	we have
		\[
	\frac{\mathcal{M}^{n} - \mathcal{M}^{n-1}}{\Delta t} +  v\cdot \nabla \mathcal{M}^{n-1} =\tau^n(g_1^n-f_1^n) - \varepsilon \left(\frac{f_1^{n}- f_1^{n-1}}{\Delta t} -v\cdot \nabla f_1^{n-1}\right) + \mathcal{O}(\Delta t).
	\]
	Then, as in Proposition 1 in \cite{FJ}, we find the explicit form of $g_1$ and $f_1$ as follows:
		\[g_1^n= \frac{\mathcal{G}^n-\mathcal{M}^n}{\varepsilon}=\left( \frac{\tau^n \rho^n}{\left(2\pi T^n\right)^{d_v/2}} + \mathcal{O}(\varepsilon^2)\right)\left(\nu \frac{\mathcal{M}^n}{2(T^n)^2}(v-u^n)\Theta_1^n(v-u^n) + \mathcal{O}(\varepsilon^2)\right),
		\]
		\[
		f_1^n=g_1^n -\frac{1}{\tau^n}\left(\frac{\mathcal{M}^{n}- \mathcal{M}^{n-1}}{\Delta t} -v\cdot \nabla \mathcal{M}^{n-1}\right) + \mathcal{O}(\varepsilon) + \mathcal{O}(\Delta t).
		\]
Consequently,
\begin{align*}
	\rho^n\Theta_1^n&=\int (v-u^n)\otimes (v-u^n) f_1^ndv\cr
%	&=\nu^n\rho^n \Theta_1^n - \frac{1}{\tau^n}\int (v-u^n)\otimes (v-u^n) \mathcal{M}^{n-1} \left[ \left(\frac{(v-u^{n-1})\otimes (v-u^{n-1})}{T^{n-1}} - \frac{|v-u^{n-1}|^2 Id}{d_v T^{n-1}}\right) : \nabla u^{n-1}\right]\cr
%	& - \frac{1}{\tau^n}\int (v-u^{n-1})\otimes (u^{n-1}-u^n) \mathcal{M}^{n-1} \left[ \left( \frac{|v-u^{n-1}|^2}{2T^{n-1}}- \frac{d_v+2}{2}\right) \frac{(v-u^{n-1})\cdot \nabla T^{n-1}}{T^{n-1}}\right]\cr
%	&- \frac{1}{\tau^n}\int (u^{n-1}-u^n)\otimes (v-u^{n-1}) \mathcal{M}^{n-1} \left[ \left( \frac{|v-u^{n-1}|^2}{2T^{n-1}}- \frac{d_v+2}{2}\right) \frac{(v-u^{n-1})\cdot \nabla T^{n-1}}{T^{n-1}}\right]\cr
%	&- \frac{1}{\tau^n}\int (u^{n-1}-u^n)\otimes (u^{n-1}-u^n) \mathcal{M}^{n-1} \left[ \left(\frac{(v-u^{n-1})\otimes (v-u^{n-1})}{T^{n-1}} - \frac{|v-u^{n-1}|^2 Id}{d_v T^{n-1}}\right) : \nabla u^{n-1}\right]\cr
%	&=\nu^n\rho^{n} \Theta_1^n - \frac{1}{\tau^n}\rho^{n-1}T^{n-1}\sigma(u^{n-1})+ 0 + 0 + \mathcal{O}((\Delta t)^2).
	&=\nu\rho^{n} \Theta_1^n - \frac{1}{\tau^n}\rho^{n-1}T^{n-1}\sigma(u^{n-1})+ \mathcal{O}(\varepsilon) + \mathcal{O}(\Delta t),\cr
	Q_1^n&=\int \frac{1}{2}|v-u^n|^2 (v-u^n) f_1^ndv\cr
	&=- \frac{1}{\tau^n}\frac{d_v+2}{2}\rho^{n-1}T^{n-1}\nabla T^{n-1} + \mathcal{O}(\varepsilon) +  \mathcal{O}(\Delta t).
\end{align*}
Therefore, the viscosity $\mu^n$ and heat conductivity $\kappa^n$ are given by
$$\mu^n=\frac{1}{(1-\nu)\tau^n}\rho^{n-1}T^{n-1}.$$
and
%$Q_1^n=-\frac{\tau^n}{1-\nu}\rho^nT^n \sigma(u^{n-1}) + \mathcal{O}(\Delta t)$
$$\kappa^n= \frac{1}{\tau^n}\frac{d_v+2}{2}\rho^{n-1}T^{n-1}.$$
This completes the proof.
%%%%%%%%%%%%%%%%%%%%%%%%%%%%%%%%%%%%%%%%%%%%%%%%%%%%%%%%%%%%%%%%%%%%%%%%%%%%%%%%%%%%%%%%%%%%%%%%%%%%%%%%%%%%%%%%%%%%%%%%%%%%%%%%%%%%%%%%%%%%%%%%%%%%%%%%%%%%%%%%%%%%%%%%%%%%%%%%%%%%%%%%%%%%%%%%%%%%%%%%%%%%

In the following two propositions, we focus on the estimates of conservation errors. For brevity of description, we denote the total mass/momentum/energy at time $t_n$ by $(m_0^n,m_1^n,m_2^n)$ and define them as follows:
$$(m_0^n,m_1^n,m_2^n):=\sum_{i,j}f_{i,j}^n \left(1,v_j,|v_j|^2/2\right) (\Delta v)^{d_v} (\Delta x)^{d_x}.$$
We start from the case of DIRK methods.
\begin{proposition}\label{prop rk}
	Assume that we adopt conservative reconstruction \cite{CBRY1} and the $L^2$-projection for the high order SL methods based on $s$-stage DIRK method \eqref{numSol RK}. 
%	If the error from the $L^2$-projection is uniformly bounded, i.e.,
%	\[
%	\left\|\sum_j \left(\mathcal{G}_{i,j}^{(\ell)}-f_{i,j}^{(\ell)}\right)\phi_j (\Delta v)^{d_v}\right\|_\infty < \delta ,\quad \ell=1,...,s
%	\]
%	for some $\delta>0$, and the numerical solutions at time $t^n$ satisfies $$\sum_{i,j}f_{i,j}^n \phi_j (\Delta v)^{d_v} (\Delta x)^{d_x}= (m_0^n,m_1^n,m_2^n),\quad \phi_j=(1,v_j,|v_j|^2/2),$$
Then, the methods are globally conservative at the discrete level on the periodic boundary condition, i.e.,
	$$\left\| (m_0^{n+1},m_1^{n+1},m_2^{n+1})-(m_0^n,m_1^n,m_2^n) \right\|_\infty=\mathcal{O}\left(\frac{\delta\Delta t}{\varepsilon}\right).$$
\end{proposition}
\begin{proof}
	Let us consider $\phi_j=1,v_j,\frac{|v_j|^2}{2}$, and take the discrete moments in \eqref{rk 1st} for $k=s$, we get
	\begin{align*}
		&\sum_{i}\sum_{j}f_{i,j}^{(s)}\phi_j (\Delta v)^{d_v} (\Delta x)^{d_x}\cr &=\sum_{i}\sum_{j}	\left(\tilde{f}_{i,j}^{(s,0)}+\sum_{\ell=1}^{s-1}a_{s\ell}\frac{\Delta t}{\varepsilon} Q_{i,j}^{(s,\ell)} +a_{ss} \frac{\tau_i^{(s)}\Delta t}{\varepsilon} \left(\mathcal{G}_{i,j}^{(s)}-f_{i,j}^{(s)}\right)\right) \phi_j (\Delta v)^{d_v} (\Delta x)^{d_x}\cr
		&=\sum_{i}\sum_{j}	\left(f_{i,j}^{n}+\sum_{\ell=1}^{s-1}a_{s\ell}\frac{\Delta t}{\varepsilon} Q_{i,j}^{(\ell)} +a_{ss} \frac{\tau_i^{(s)}\Delta t}{\varepsilon} \left(\mathcal{G}_{i,j}^{(s)}-f_{i,j}^{(s)}\right)\right) \phi_j (\Delta v)^{d_v} (\Delta x)^{d_x}\cr
		&=\sum_{i}\sum_{j}	\left(f_{i,j}^{n}+\sum_{\ell=1}^{s}a_{s\ell}\frac{\Delta t}{\varepsilon} Q_{i,j}^{(\ell)} \right) \phi_j (\Delta v)^{d_v} (\Delta x)^{d_x}\cr
		&=I_1+I_2,
	\end{align*}
	where the second equality holds because the telescoping summations of first and second terms are preserved on the periodic boundary condition owing to the use of conservative reconstruction \cite{CBRY1}. Note that $$I_1=(m_0^n,m_1^n,m_2^n).$$ 
	In case of $I_2$, the use of $L^2$-projection implies that
	\begin{align*}
		\|I_2\|_\infty&=\left\| \sum_{i} \sum_{\ell=1}^s a_{s\ell} \frac{\tau_i^{(\ell)}\Delta t}{\varepsilon} \sum_{j} \left(\mathcal{G}_{i,j}^{(\ell)}-f_{i,j}^{(\ell)}\right)\phi_j (\Delta v)^{d_v} (\Delta x)^{d_x} \right\|_\infty\cr
		&\leq \max_{i,\ell}\frac{|a_{s\ell}\tau_i^{(\ell)}\Delta t|}{\varepsilon} \delta   (x_R-x_L)^{d_x},
	\end{align*}
where $\delta$ means the machine precision (see Remark \ref{rem delta}). Recalling the SA property of DIRK methods, we have  $f_{i,j}^{(s)}=f_{i,j}^{n+1}$. This completes the proof.
\end{proof}
Next, we provide conservation error estimates for BDF methods:
\begin{proposition}\label{prop bdf}
	Assume that we adopt conservative reconstruction \cite{CBRY1} and the $L^2$-projection for the high order SL methods based on $s$-step BDF method \eqref{numSol bdf}. 
%	If the error from the $L^2$-projection is uniformly bounded, i.e.,
%	\[
%	\left\|\sum_j \left(\tilde{\mathcal{G}}_{i,j}^{n}-\tilde{f}_{i,j}^{n}\right)\phi_j (\Delta v)^{d_v}\right\|_\infty < \delta ,\quad \ell=1,...,s
%	\]
%	for some $\delta>0$, and the numerical solutions at time $t^n$ satisfies $$\sum_{i,j}f_{i,j}^{n+1-k} \phi_j (\Delta v)^{d_v} (\Delta x)^{d_x}= (m_0^{n+1-k},m_1^{n+1-k},m_2^{n+1-k}),\quad \phi_j=(1,v_j,|v_j|^2/2),\quad k=1,...,s$$
Then, the methods are globally conservative at the discrete level on the periodic boundary condition, i.e.,
	$$\left\| (m_0^{n+1},m_1^{n+1},m_2^{n+1})-\sum_{k=1}^s \alpha_k 
	(m_0^{n+1-k},m_1^{n+1-k},m_2^{n+1-k}) \right\|_\infty=\mathcal{O}\left(\frac{\delta\Delta t}{\varepsilon}\right).$$
\end{proposition}

\begin{proof}
	As in the proof of Proposition \ref{prop rk}, we consider $\phi_j=1,v_j,\frac{|v_j|^2}{2}$, and take the discrete moments in \eqref{bdf form} to obtain
	\begin{align*}
		\sum_{i}\sum_{j}f_{i,j}^{n+1}\phi_j (\Delta v)^{d_v} (\Delta x)^{d_x} 
%		&=\sum_{i}\sum_{j}	\left(\tilde{f}_{i,j}^{n}+\beta_{s} \frac{\tau_i^{n+1}\Delta t}{\varepsilon} \left({\mathcal{G}}_{i,j}^{n+1}-f_{i,j}^{n+1}\right)\right) \phi_j (\Delta v)^{d_v} (\Delta x)^{d_x}\cr
		&=\sum_{i}\sum_{j}	\left( \sum_{k=1}^s \alpha_k f_{i,j}^{n,k} +\beta_{s} \frac{\tau_i^{n+1}\Delta t}{\varepsilon} \left({\mathcal{G}}_{i,j}^{n+1}-f_{i,j}^{n+1}\right)\right) \phi_j (\Delta v)^{d_v} (\Delta x)^{d_x}\cr
		&=\sum_{i}\sum_{j}	\left( \sum_{k=1}^s \alpha_k f_{i,j}^{n+1-k} +\beta_{s} \frac{\tau_i^{n+1}\Delta t}{\varepsilon} \left({\mathcal{G}}_{i,j}^{n+1}-f_{i,j}^{n+1}\right)\right) \phi_j (\Delta v)^{d_v} (\Delta x)^{d_x}
	\end{align*}
	where the last equality holds because the telescoping summations of the first term is preserved on the periodic boundary condition due to the use of conservative reconstruction \cite{CBRY1}. 
%	If we use the fact that $\tilde{f}_{i,j}^n$ and $f_{i,j}^{n+1}$ share the zeroth, first and second moments w.r.t. $\phi_j$, we have
Then, we have
	\begin{align*}
		&\sum_{i}\sum_{j}\left(f_{i,j}^{n+1}-\sum_{k=1}^s \alpha_k f_{i,j}^{n+1-k}\right)\phi_j (\Delta v)^{d_v} (\Delta x)^{d_x}\cr
		&=\sum_{i}\sum_{j}	\beta_{s} \frac{\tau_i^{n+1}\Delta t}{\varepsilon} \left({\mathcal{G}}_{i,j}^{n+1}-f_{i,j}^{n+1}\right) \phi_j (\Delta v)^{d_v} (\Delta x)^{d_x}.
	\end{align*}
 By the use of $L^2$-projection, the $L^\infty$-norm of right hand side is machine precision $\delta$. This implies the desired result.
\end{proof}

%Since $f_{i,j}^{(k)}= f_{i,j}^{0}$, by changing the order of summation we have
%$$I_1=\sum_{i}\left[\left(\sum_{j}f_{i,j}^{(k)} \phi_j (\Delta v)^{d_v}\right)  + err_i\right] (\Delta x)^{d_x}.$$
%Similarly, we obtain
%\begin{align*}
%	I_2=\sum_{i}\sum_{j}\frac{\tau_i^s a_{ss} \Delta t\mathcal{G}_i^{(s)}}{\varepsilon + \tau_i^s a_{ss}\Delta t} \phi_j (\Delta v)^{d_v} (\Delta x)^{d_x}\cr
%\end{align*}

\begin{proposition}
	%	The proposed methods satisfy	
	%	\begin{enumerate}
		%		\item 
		For all $\Delta t> 0$ and initial data $f^0$, the distribution function $f^n$ obtained by the proposed high order semi-Lagrangian methods converges to $\mathcal{M}^n$, i.e.,
		$\lim\limits_{\varepsilon \to 0} f^n = \mathcal{M}^n.$
		%	\end{enumerate}
	%	
\end{proposition}

\begin{proof}
	\noindent For simplicity, we only prove the case of second order methods, DIRK2 and BDF2.\\
	$\bullet$ \textbf{Proof for DIRK2 method}:	
	For $k=1$, we have
	%		 \begin{align*}
		%		 	\Sigma_i^{(1)}&= \frac{\varepsilon}{\varepsilon + (1-\nu)\tau_i^{(1)} a_{11}\Delta t}\tilde{\Sigma}_i^{(1)}  + \frac{(1-\nu)\tau_i^{(1)} a_{11}\Delta t}{\varepsilon + (1-\nu)\tau_i^{(1)} a_{11}\Delta t}\rho_i^{{(1)}}\left(  T_i^{(1)}Id + u_i^{(1)}\otimes u_i^{(1)}\right) \cr
		%		 	&=\rho_i^{(1)}\left(  T_i^{(1)}Id + u_i^{(1)}\otimes u_i^{(1)}\right) +\frac{\varepsilon}{\varepsilon + (1-\nu)\tau_i^{(1)} a_{11}\Delta t}\left(\tilde{\Sigma}_i^{(1)}- \rho_i^{{(1)}}\left(  T_i^{(1)}Id + u_i^{(1)}\otimes u_i^{(1)}\right)\right).
		%		 \end{align*}
	%		 This further implies that  
	%		 $$\rho_i^{(1)}\Theta_i^{(1)} = \left[\Sigma_i^{(1)} - \rho_i^{(1)} u_i^{(1)}\otimes u_i^{(1)}\right] = \rho_i^{(1)} T_i^{(1)}Id ++\frac{\varepsilon}{\varepsilon + (1-\nu)\tau_i^{(1)} a_{11}\Delta t}\left(\tilde{\Sigma}_i^{(1)}- \rho_i^{{(1)}}\left(  T_i^{(1)}Id + u_i^{(1)}\otimes u_i^{(1)}\right)\right)$$ 
	%		 and hence
	\begin{align*}
		\mathcal{T}_i^{(1)}
		&=(1-\nu_i^{(1)})T_i^{(1)} I + \nu_i^{(1)} \left(\frac{\tilde{\Sigma}_i^{(1)}}{\rho_i^{(1)}}- u_i^{(1)}\otimes u_i^{(1)} \right)\cr
		&=\left(1- \frac{\varepsilon \nu}{\varepsilon + (1-\nu)\tau_i^{(1)} a_{11}\Delta t}\right)T_i^{(1)} I + \frac{\varepsilon \nu}{\varepsilon + (1-\nu)\tau_i^{(1)} a_{11}\Delta t}\left(\frac{\tilde{\Sigma}_i^{(1)}}{\rho_i^{(1)}}- u_i^{(1)}\otimes u_i^{(1)} \right)\cr
		&=T_i^{(1)} I + \frac{\varepsilon \nu}{\varepsilon + (1-\nu)\tau_i^{(1)} a_{11}\Delta t}\left(\frac{\tilde{\Sigma}_i^{(1)}}{\rho_i^{(1)}}- u_i^{(1)}\otimes u_i^{(1)} - T_i^{(1)} I \right).
	\end{align*}
	Then, we get
	\begin{align*}
		\mathcal{T}_i^{(1)}
		&=T_i^{(1)} I + \mathcal{O}(\varepsilon).
	\end{align*}
	In the limit $\varepsilon\to 0$, we obtain
	$$\lim\limits_{\varepsilon\to 0}\mathcal{G}_{i,j}^{(1)}=\mathcal{M}_{i,j}^{(1)}=
	\frac{\rho_i^{(1)}}{(2 \pi T_i^{(1)})^{\frac{d_v}{2}}} \exp \left(-\frac{|v_j-U_i^{(1)}|^2}{2T_i^{(1)}}\right)=:\mathcal{M}_{i,j}^{(1)}.$$
	We further note that
	\begin{align*}
		f_{i,j}^{(1)} - \mathcal{G}_{i,j}^{(1)} &=\frac{\varepsilon\tilde{f}_{i,j}^{(1)} +   \tau_i^{(1)} a_{11} \Delta t\mathcal{G}_i^{(1)}}{\varepsilon + \tau_i^{(1)} a_{11}\Delta t}-\mathcal{G}_{i,j}^{(1)}\cr
		&=\frac{\varepsilon}{\varepsilon + \tau_i^{(1)} a_{11}\Delta t}(\tilde{f}_{i,j}^{(1)} - \mathcal{G}_i^{(1)})\cr
		&=\mathcal{O}(\varepsilon).
	\end{align*}
	Then, $Q_{i,j}^{(1)}=\tau_{i,j}^{(1)}\left(f_{i,j}^{(1)} - \mathcal{G}_{i,j}^{(1)}\right)=\mathcal{O}(\varepsilon).$
	Now we move on to the second stage value:
	\begin{align*}
		f_{i,j}^{(2)} &=\frac{\varepsilon\tilde{f}_{i,j}^{(2)} +   \tau_i^{(2)} a_{22} \Delta t\mathcal{G}_i^{(2)}}{\varepsilon + \tau_i^{(2)} a_{22}\Delta t}.
	\end{align*}	
	As in the first stage, we have
	\begin{align*}
		\mathcal{T}_i^{(2)}
		&=(1-\nu_i^{(2)})T_i^{(2)} I + \nu_i^{(2)} \left(\frac{\tilde{\Sigma}_i^{(2)}}{\rho_i^{(2)}}- u_i^{(1)}\otimes u_i^{(2)} \right)\cr
		&=\left(1- \frac{\varepsilon \nu}{\varepsilon + (1-\nu)\tau_i^{(2)} a_{22}\Delta t}\right)T_i^{(2)} I + \frac{\varepsilon \nu}{\varepsilon + (1-\nu)\tau_i^{(2)} a_{22}\Delta t}\left(\frac{\tilde{\Sigma}_i^{(2)}}{\rho_i^{(2)}}- u_i^{(2)}\otimes u_i^{(2)} \right)\cr
		&=T_i^{(2)} I + \frac{\varepsilon \nu}{\varepsilon + (1-\nu)\tau_i^{(2)} a_{22}\Delta t}\left(\frac{\tilde{\Sigma}_i^{(2)}}{\rho_i^{(2)}}- u_i^{(2)}\otimes u_i^{(2)} - T_i^{(2)} I \right).
	\end{align*}
	Since $\tilde{f}_{i,j}^{(2)}= \tilde{f}_{i,j}^{(2,0)}+a_{21}\frac{\Delta t}{\varepsilon} Q_{i,j}^{(2,1)}=\mathcal{O}(1)$, $\tilde{\Sigma}_i^{(2)}=\mathcal{O}(1)$. As a consequence, we get
	\begin{align*}
		\mathcal{T}_i^{(2)}
		&=T_i^{(2)} I + \mathcal{O}(\varepsilon).
	\end{align*}
	Thus, in the limit $\varepsilon \to 0$, the Gaussian function converges to isotropic Maxwellian 
	$$\lim\limits_{\varepsilon\to 0}\mathcal{G}_{i,j}^{(2)}=
	\frac{\rho_i^{(2)}}{(2 \pi T_i^{(2)})^{\frac{d_v}{2}}} \exp \left(-\frac{|v_j-U_i^{(2)}|^2}{2T_i^{(2)}}\right)=:\mathcal{M}_{i,j}^{(2)}.$$
	To sum up, 
	\begin{align*}
		\lim\limits_{\varepsilon\to 0}f_{i,j}^{(2)} &=\lim\limits_{\varepsilon\to 0} \frac{\varepsilon\tilde{f}_{i,j}^{(2)} +   \tau_i^{(2)} a_{22} \Delta t\mathcal{G}_i^{(2)}}{\varepsilon + \tau_i^{(2)} a_{22}\Delta t}= \mathcal{M}_{i,j}^{(2)},
	\end{align*}	
	which together with the SA property of the DIRK method implies the desired result.

	\noindent $\bullet$ \textbf{Proof of BDF2 methods}: The asymptotic limit $\varepsilon \to 0$ gives
	\begin{align*}
		\lim\limits_{\varepsilon \to 0} \Sigma_i^{n+1}&=\lim\limits_{\varepsilon \to 0}\left[\frac{\varepsilon}{\varepsilon + (1-\nu)\tau_i^{n+1} \beta_s\Delta t}\tilde{\Sigma}_i^{n}  + \frac{(1-\nu)\tau_i^{n+1} \beta_s\Delta t}{\varepsilon + (1-\nu)\tau_i^{n+1} \beta_s\Delta t}\rho_i^{n+1}\left(  T_i^{n+1}I + u_i^{n+1}\otimes u_i^{n+1}\right)\right]\cr
		&=\rho_i^{n+1}\left(  T_i^{n+1}I + u_i^{n+1}\otimes u_i^{n+1}\right).
	\end{align*}
	This further implies that  
	$$ \lim\limits_{\varepsilon \to 0}\rho_i^{n+1}\Theta_i^{n+1} = \lim\limits_{\varepsilon \to 0}\left[\Sigma_i^{n+1} - \rho_i^{n+1} u_i^{n+1}\otimes u_i^{n+1}\right] = \rho_i^{n+1} T_i^{n+1}I$$ 
	and hence
	\begin{align*}
		\lim\limits_{\varepsilon \to 0}\mathcal{T}_i^{n+1}
		&=T_i^{n+1} I.
	\end{align*}
	To sum up, in the limit $\varepsilon \to 0$, the distribution function converges to isotropic Maxwellian 
	\[
	\mathcal{M}_{i,j}^{n+1}=
	\frac{\rho_i^{n+1}}{(2 \pi T_i^{n+1})^{\frac{d_v}{2}}} \exp \left(-\frac{|v_j-U_i^{n+1}|^2}{2T_i^{n+1}}\right).
	\]
This completes the proof.	
\end{proof}

%%%%%%%%%%%%%%%%%%%%%%%%%%%%%%%%%%%%%%%%%%%%%%%%%%%%%%%%%%%%%%%%%%%%%%%%%%%%%%%%%%%%%%%%%%%%%%%%%%%%%%%%%%%%%%%%%%%%%%%%%%%%%%%%%%%%%%%%%%%%%%%%%%%%%%%%%%%%%%%%%%%%%%%%%%%%%

\section{Numerical tests}\label{sec numerical test}

Throughout this section, we consider the truncated velocity domain $[-v_{max},v_{max}]^{d_v}$ ($d_v=2,3$) with $N_v+1$ uniform grids. We restrict our numerical tests into 1D space dimension. Hence, we use the time step $\Delta t$ that corresponds to CFL number defined by
\begin{align}\label{CFL}
%	(1D):\quad 
	\text{CFL}:&= v_{max}\frac{\Delta t}{\Delta x}.
%	(2D):\quad \text{CFL}:&= \frac{v_{max}}{\Delta x} + \frac{v_{max}}{\Delta y}.
\end{align}
The choice of free parameter $\nu$ and relaxation term $\tau$ will be explained in each problem.

\subsection{Accuracy test}
In this test, we confirm the accuracy of the proposed high order method for ES-BGK model of one dimension in space and two dimensions in velocity. For this, we consider smooth profile of the solution adopted in \cite{PP2}. We take the Maxwellian as initial data, which corresponds to the following macroscopic variables :
\begin{align}\label{init accuracy}
	\begin{split}
		\rho_0&=1,\quad u_0=\left(\frac{1}{\sigma}\big(\exp\left(-(\sigma x-1)^2\right)-2\exp(-(\sigma x+3)^2)\big),0\right),\quad T_0=1.
	\end{split}
\end{align}
We consider uniform grid $\Delta x$ on the periodic interval $[-1,1]$, and truncated velocity domain $[-10,10]^2$ with $N_v=32$. That is, we use $33^2$ velocity grids. The numerical solutions are computed up to $t=0.32$ because shock appears around $t\approx 0.35$ for small Knudsen number. The time step is fixed according to the CFL condition \ref{CFL} with CFL$=4$. We use the parameter $\nu=-1$, and relaxation term $\tau=1$ for simplicity. Note that the ES-BGK model is well-defined for $\nu=-1$ when $d_v=2$.

\begin{center}
	\begin{table}[htbp]
		\centering	
		{\begin{tabular}{|cccccccccccc|}
				\hline
				\multicolumn{10}{ |c| }{Relative $L^1$ error and order of density} \\ \hline
				\multicolumn{1}{ |c }{}&
				\multicolumn{1}{ |c| }{}& \multicolumn{2}{ c  }{$\varepsilon=10^{-4}$} & \multicolumn{2}{ |c }{$\varepsilon=10^{-3}$}& \multicolumn{2}{ |c| }{$\varepsilon=10^{-2}$} &
				\multicolumn{2}{ |c| }{$\varepsilon=10^{-1}$}&
				\multicolumn{2}{ |c| }{$\varepsilon=10^{-0}$} \\ \hline
				\multicolumn{1}{ |c }{}&
				\multicolumn{1}{ |c|  }{$(N_x,2N_x)$} &
				\multicolumn{1}{ c  }{error} &
				\multicolumn{1}{ c|  }{rate} &
				\multicolumn{1}{ |c  }{error} &
				\multicolumn{1}{ c  }{rate} &
				\multicolumn{1}{ |c  }{error} &
				\multicolumn{1}{ c|  }{rate} &
				\multicolumn{1}{ |c  }{error} &
				\multicolumn{1}{ c|  }{rate} &
				\multicolumn{1}{ |c  }{error} &
				\multicolumn{1}{ c|  }{rate}     \\ 
				\hline
				\hline	
				\multicolumn{1}{ |c }{}&
				\multicolumn{1}{ |c|  }{$(80,160)$}&4.292e-04
				
				&2.63
				&3.353e-04
				
				&2.73&1.111e-04
				
				&2.79&3.079e-05
				
				&2.74&2.174e-04
				
				&2.74
				\\
				\multicolumn{1}{ |c }{RK2}&
				\multicolumn{1}{ |c|  }{$(160,320)$}&6.943e-05
				
				&2.68 
				&5.059e-05
				
				&2.75&1.602e-05
				
				&2.66&4.617e-06
				
				&2.90&3.258e-05
				
				&2.95
				\\
				\multicolumn{1}{ |c }{}&
				\multicolumn{1}{ |c|  }{$(320,640)$}&1.083e-05
				&     
				&7.504e-06
				&&2.535e-06
				&&6.188e-07
				&&4.217e-06
				&
				\\
				\hline
				\hline
				\multicolumn{1}{ |c }{}&
				\multicolumn{1}{ |c|  }{$(80,160)$}&9.313e-04
				
				&2.25
				&8.188e-04
				
				&2.29&3.012e-04
				
				&2.24&5.799e-05
				
				&2.46&3.437e-04
				
				&2.66
				\\
				\multicolumn{1}{ |c }{BDF2}&\multicolumn{1}{ |c|  }{$(160,320)$}&1.951e-04
				
				&2.35     
				&1.671e-04
				
				&2.33&6.383e-05
				
				&2.19&1.053e-05
				
				&2.38&5.454e-05
				
				&2.93
				\\
				\multicolumn{1}{ |c }{}&
				\multicolumn{1}{ |c|  }{$(320,640)$}&3.816e-05
				&     
				&3.327e-05
				&&1.3978e-05
				&&2.019e-06
				&&7.141e-06
				&
				\\
				\hline
				\hline
				\multicolumn{1}{ |c }{}&
				\multicolumn{1}{ |c|  }{$(80,160)$}&2.015e-04
				
				&2.42
				&1.399e-04
				
				&2.54&2.139e-05
				
				&2.75&5.024e-06
			
				&4.74&3.434e-05
				
				&4.80
				\\
				\multicolumn{1}{ |c }{RK3}&
				\multicolumn{1}{ |c|  }{$(160,320)$}&3.753e-05
				
				&2.08     
				&2.406e-05
				
				&2.28&3.176e-06
				
				&2.74&	1.875e-07
				
				&4.48&1.236e-06
				
				&4.97
				\\
				\multicolumn{1}{ |c }{}&
				\multicolumn{1}{ |c|  }{$(320,640)$}&8.883e-06
				&
				%2.10
				&4.940e-06
				&
				%2.40
				&4.762e-07
				&
				%2.84
				&8.380e-09
				&
				%3.42
				&3.943e-08
				&
				%5.02
				
				\\
%				\multicolumn{1}{ |c }{}&
%\multicolumn{1}{ |c|  }{$(640,1280)$}&2.077e-06
%
%&2.16
%&9.388e-07
%
%&2.54&6.654e-08
%
%&2.92&7.835e-10
%
%&3.15&1.212e-09
%
%&5.06
%\\
%				\multicolumn{1}{ |c }{}&
%\multicolumn{1}{ |c|  }{$(1280,2560)$}&4.659e-07
%&     
%&1.617e-07
%&&8.791e-09
%&&8.833e-11
%&&3.624e-11
%&
%\\
				\hline
				\hline
				\multicolumn{1}{ |c }{}&
				\multicolumn{1}{ |c|  }{$(80,160)$}&4.734e-04
				
				&2.32    
				&3.580e-04
				
				& 2.28
				
				&1.051e-04
				&2.17
				&1.395e-05
				
				&3.81
				&7.693e-05
				
				&4.75
				
				\\
				\multicolumn{1}{ |c }{BDF3}&
				\multicolumn{1}{ |c|  }{$(160,320)$}&9.483e-05
				
				&2.70     
				&7.368e-05
				
				&2.56&2.345e-05
				
				&2.54&9.936e-07
				
				&3.05&2.854e-06
				
				&5.02
				\\
				\multicolumn{1}{ |c }{}&
				\multicolumn{1}{ |c|  }{$(320,640)$}&1.464e-05  
				&
				%2.86
				&1.251e-05
				&
				%2.37
				&4.029e-06
				&
				%2.74
				&1.200e-07&
				%2.99
				&8.793e-08& %4.90
				
				\\
%				\multicolumn{1}{ |c }{}&
%				\multicolumn{1}{ |c|  }{$(640,1280)$}&2.022e-06
%				  
%				&2.79&2.420e-06
%				
%				&2.20&6.023e-07
%				&2.86
%				&1.513e-08
%				
%				&2.99&2.936e-09
%				
%				&3.54
%				\\
%				\multicolumn{1}{ |c }{}&
%				\multicolumn{1}{ |c|  }{$(1280,2560)$}&  2.914e-07
%				&&5.282e-07
%				&&8.294e-08
%				&&1.905e-09
%				&&2.517e-10
%				&
%				\\
				\hline
		\end{tabular}}
		\caption{Accuracy test for the x1d-v2d ES-BGK equation. Initial data is given in \eqref{init accuracy}. 
		}\label{2d accuracy}
	\end{table}
\end{center}

In Table \ref{2d accuracy}, we report relative $L^1$ error of density and corresponding order of convergence. We observe the expected order of convergece for each scheme. For example, RK2 and BDF2 based methods are combined with QCWENO23, in which we can expect second order in time and third order in space. The numerical results correspond to the expected ones. In case of RK3 and BDF3 methods, we adopt QCWENO35. Both schemes also attain expected order of convergence, third order in time and fifth order in space. Although we observe order reductions for RK3 for small Knudsen numbers, it is a typical problem of DIRK3 method. %In case of BDF3, order reduction occurs due to the use of the smae DIRK3 for initializations.

Note that for relatively large Knudsen numbers time error becomes negligible compared to spatial errors because the collision term is very small. In this regime, we observe spatial errors.
In contrast, for relatively small Knudsen numbers time error becomes dominant compared to spatial errors as collision term becomes larger.

\subsection{Riemann problem}
In this problem, we consider a Riemann problem in 1D space and 2D velocity domain. The same test has been adopted in \cite{FJ} to show the consistency between the ES-BGK model and BTE or NSE. As initial macroscopic states, we use
\begin{align*}
	(\rho_0,\, u_x,\, u_y,\, T)=
	\begin{cases}
		(1,M\sqrt{2},0,1),\quad \text{if} \quad -1\leq x\leq0.5\\
		\left(\frac{1}{8},0,0,\frac{1}{4}\right),\quad \text{otherwise}
	\end{cases}.
\end{align*}
where $M=2.5$ is the Mach number. Here we impose the free-flow boundary condition in space $x\in [-1,2]$. We truncate the velocity domain with $v_{max}=15$ and take the final time as $T_f=0.4$. We use the time step that corresponds to CFL$=2$. Here we use $\nu=-1$ and $\tau=0.9\pi\rho/2$ following \cite{FJ}.
This implies that the viscosity and heat conductivity are given by
\[
\mu=\frac{1}{0.9\pi}T,\quad \kappa=\frac{4}{0.9\pi}T.
\]
We use this transport coefficients for NSE. We remark that this choice matches the viscosity of ES-BGK model and the one derived from BTE for Maxwellian molecules \cite{FJ}.

\begin{figure}[h]
	\centering
	\begin{subfigure}[h]{0.49\linewidth}
		\includegraphics[width=1\linewidth]{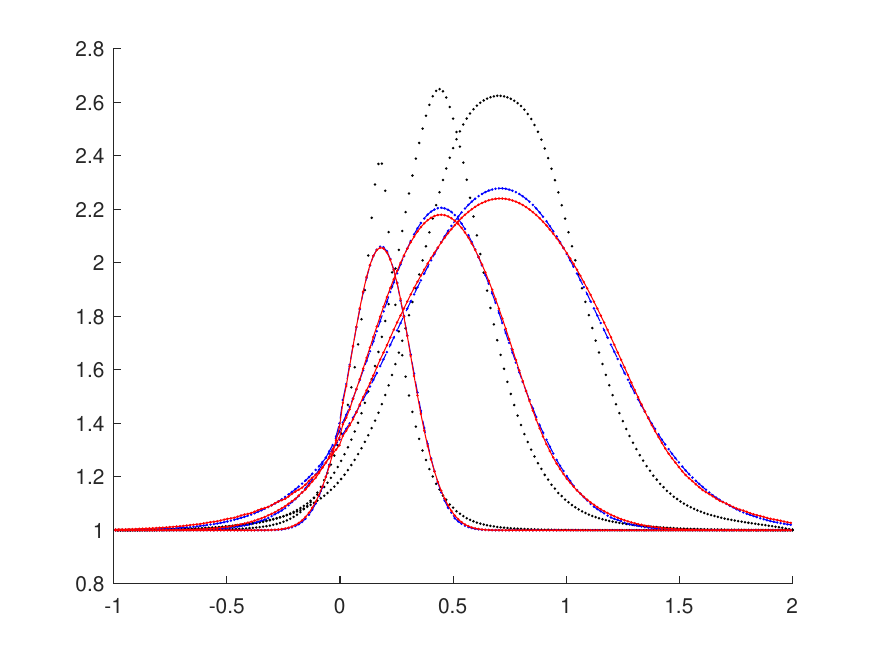}
		\subcaption{$\rho$}
	\end{subfigure}	
	\begin{subfigure}[h]{0.49\linewidth}
		\includegraphics[width=1\linewidth]{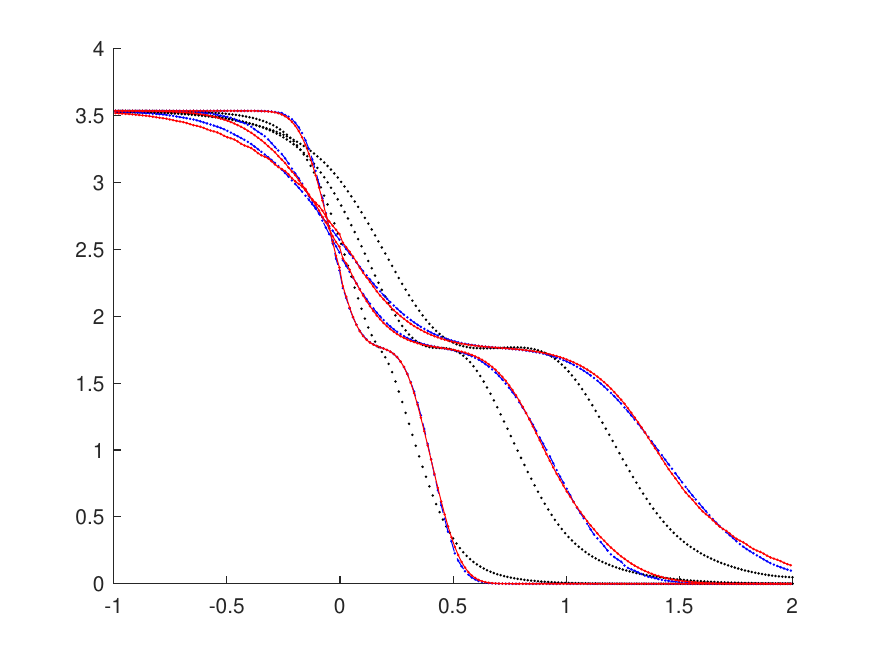}
		\subcaption{$u_1$}
	\end{subfigure}	
	\begin{subfigure}[h]{0.49\linewidth}
		\includegraphics[width=1\linewidth]{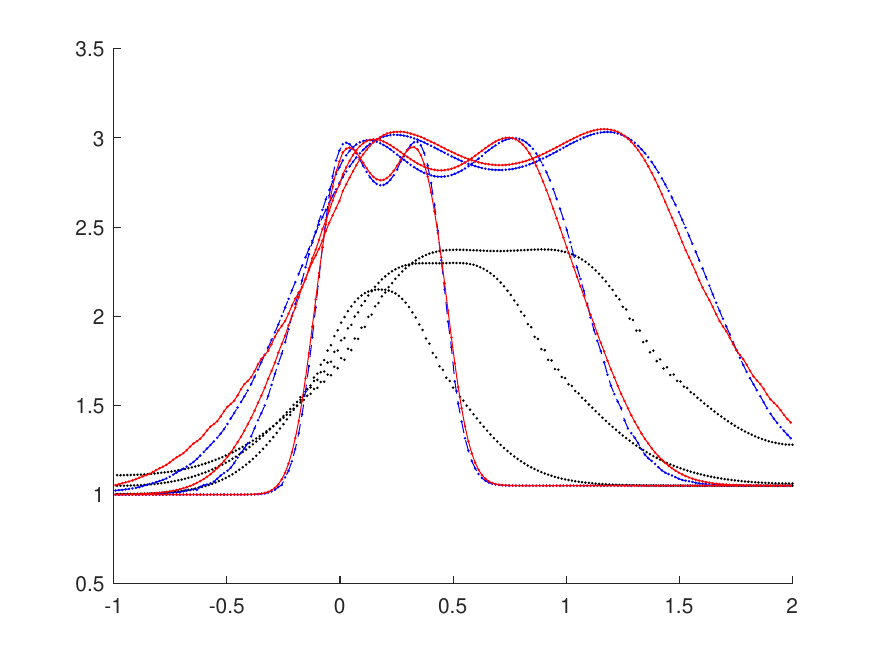}
		\subcaption{$T$}
	\end{subfigure}	
	\begin{subfigure}[h]{0.49\linewidth}
		\includegraphics[width=1\linewidth]{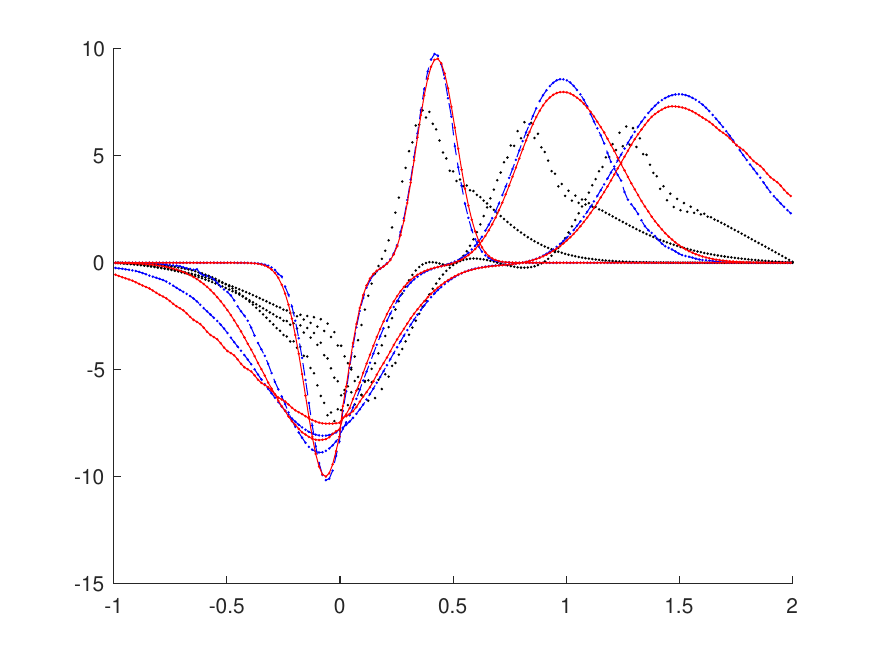}
		\subcaption{$Q$}
	\end{subfigure}	
	\caption{A Riemann problem. Comparison of three solutions: BTE(blue) ES-BGK(red) NSE(black) for $\varepsilon=0.5$, CFL=$2$, $N_x=200$, $N_v=160$.}\label{fig2}
\end{figure}		
\begin{figure}[h]	
	\centering
	\begin{subfigure}[h]{0.49\linewidth}
		\includegraphics[width=1\linewidth]{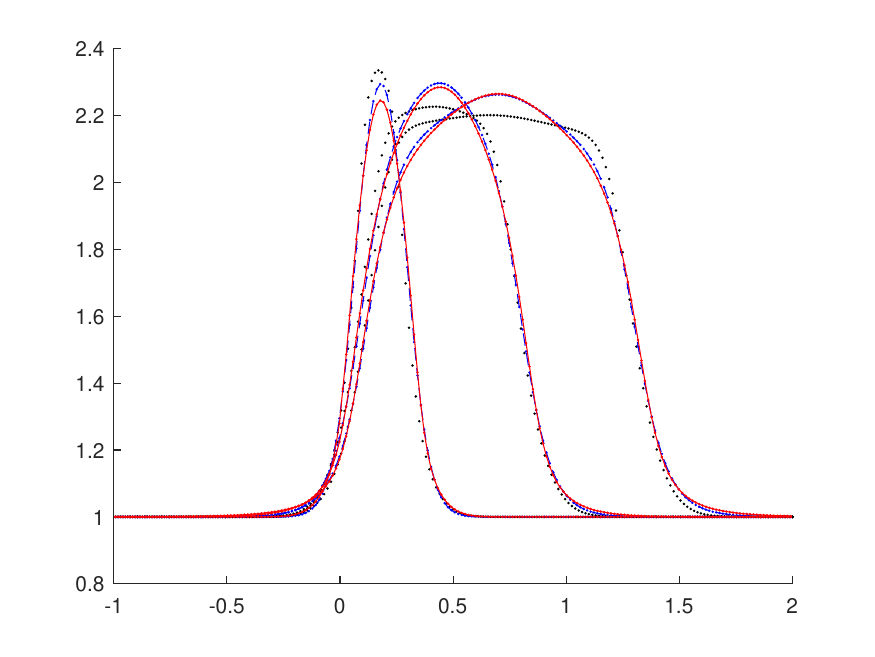}
		\subcaption{$\rho$}
	\end{subfigure}	
	\begin{subfigure}[h]{0.49\linewidth}
		\includegraphics[width=1\linewidth]{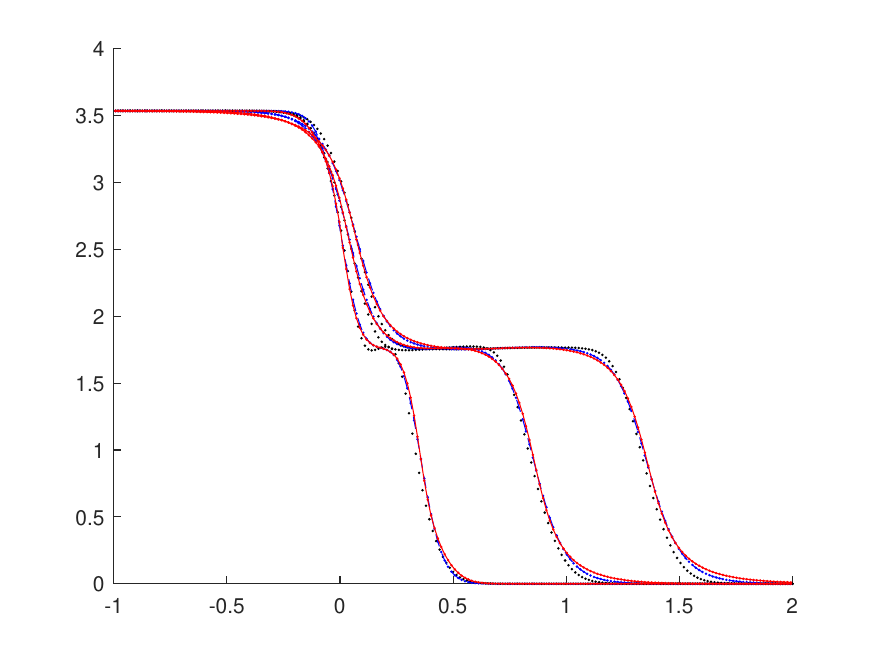}
		\subcaption{$u_1$}
	\end{subfigure}	
	\begin{subfigure}[h]{0.49\linewidth}
		\includegraphics[width=1\linewidth]{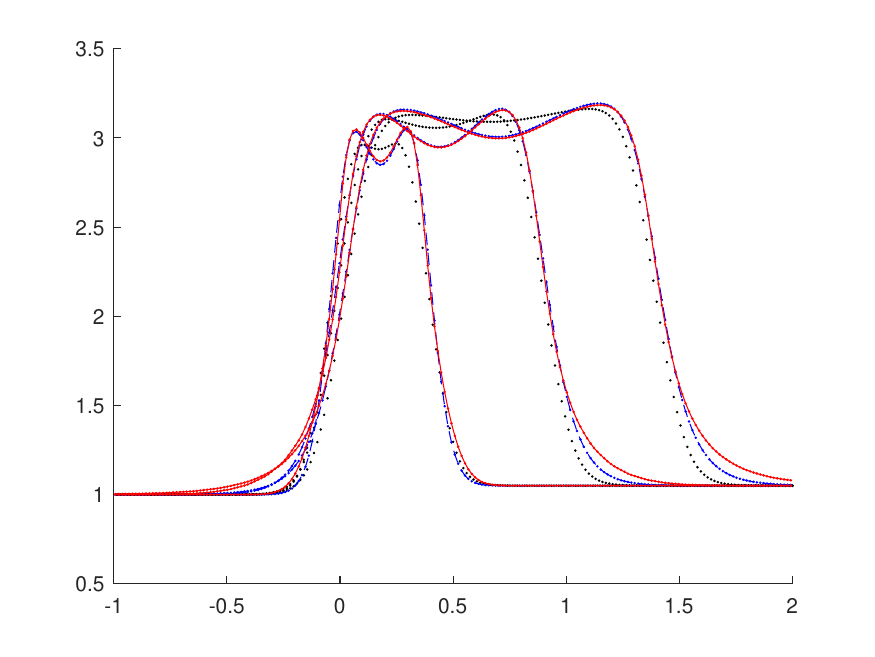}
		\subcaption{$T$}
	\end{subfigure}	
	\begin{subfigure}[h]{0.49\linewidth}
		\includegraphics[width=1\linewidth]{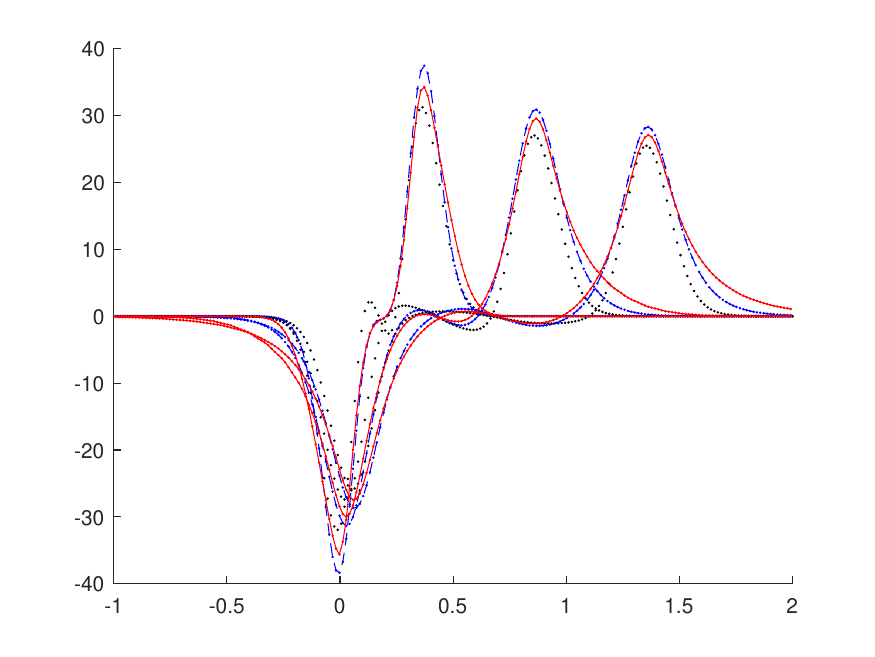}
		\subcaption{$Q$}
	\end{subfigure}	
	\caption{A Riemann problem. Comparison of three solutions: BTE(blue), ES-BGK(red), NSE(black) for $\varepsilon=0.1$, CFL=$2$, $N_x=200$, $N_v=96$.}\label{fig3}
	
%	\centering
%	\begin{subfigure}[h]{0.24\linewidth}
%		\includegraphics[width=1\linewidth]{figures/test2_Kn0.001_comparison_rho}
%		\subcaption{$\rho$}
%	\end{subfigure}	
%	\begin{subfigure}[h]{0.24\linewidth}
%		\includegraphics[width=1\linewidth]{figures/test2_Kn0.001_comparison_u1}
%		\subcaption{$u_1$}
%	\end{subfigure}	
%	\begin{subfigure}[h]{0.24\linewidth}
%		\includegraphics[width=1\linewidth]{figures/test2_Kn0.001_comparison_T}
%		\subcaption{$T$}
%	\end{subfigure}	
%	\begin{subfigure}[h]{0.24\linewidth}
%		\includegraphics[width=1\linewidth]{figures/test2_Kn0.001_comparison_Q}
%		\subcaption{$Q$}
%	\end{subfigure}	
%	\caption{A Riemann problem. Comparison of three solutions: BTE(blue) ES-BGK(red) NSE(black) for $\varepsilon=0.001$, CFL=$2$, $N_x=200$, $N_v=160$.}
\end{figure}	
In Figure \ref{fig2}, when Knudsen number is large $\varepsilon=0.5$, numerical solutions of ES-BGK and BTE are close to each other. On the other hand, the NSE solution is very far from the others. As Knudser number becomes smaller, however, in Figure \ref{fig3} it appears that the three solutions are getting closer to each other. This confirms the results of Chapmann-Enskog expansion.
\subsection{Lax shock tube problem}
In this test, we consider 1D-3D Lax shock tube problem in \cite{HZ}. For this, we set the relaxation term $\tau$ for ES-BGK model as
$$\tau=\frac{\rho T}{(1-\nu)\mu}=\frac{2}{3}\rho \sqrt{T}.$$ 
which yields the viscosity and heat conductivity:
$$\mu= \sqrt{T},\quad\kappa= \frac{d_v+2}{2}\mu (1-\nu)=\frac{15}{4}\sqrt{T}.$$
Here the free parameter $\nu=-1/2$ is taken. This implies that the Prandtl number is given by $Pr=\frac{d_v+2}{2}\frac{\mu}{\kappa}=\frac{1}{1-\nu}=\frac{2}{3}$. For initial macroscopic variables, we consider
\begin{align*}
	(\rho_0,\, u_x,\, u_y,\, u_z,\, T)=
	\begin{cases}
		(0.445, 0.698, 0, 0, 3.528),\quad \text{if} \quad -5\leq x\leq 0\\
		\left(0.5, 0, 0, 0, 0.571\right),\quad 0 < x\leq 5
	\end{cases}.
\end{align*}
on the free-flow condition $x\in [-5,5]$ and the truncated velocity domain with $v_{max}=20$.
We use the local Maxwellian as initial data.
%\[
%f_0 = \mathcal{M}.
%\]
%We take well-prepared initial data as follows:
%\[
%f_0 = \mathcal{M}- \frac{\varepsilon}{\tau}(I-\Pi_\mathcal{M})(v_1\partial_x \mathcal{M}).
%\]
%For a given $f$, the projection operator $\Pi_\mathcal{M}$ is defined by
%$$
%\Pi_\mathcal{M}(f) = \frac{1}{\rho}\biggl[ \langle f \rangle + \frac{(v-u) \cdot \langle (v-u)f \rangle }{T} + \left( \frac{|v-u^2|}{2T}- \frac{d_v}{2}\right)\frac{2}{d_v} \biggl\langle \left( \frac{|v-u^2|}{2T}- \frac{d_v}{2}\right)f \biggr\rangle \biggr]\mathcal{M}
%$$
%where the macroscopic variables $\rho$, $u$, $T$ are obtained by $f$, $d_v=3$ in this test, and $\langle \phi(v) \rangle := \int_{\mathbb{R}^{d_v}} \phi(v)\,dv$.
%The details on the properties of the operator, we refer to \cite{BLM,HZ}.
  
In Figure \ref{fig lax2}, we compare the numerical solutions to ES-BGK model computed by $N_x=200$, $N_v=40$, CFL$=2$ with reference solutions to Navier-Stokes equations. We observe that macroscopic variables for ES-BGK model and NSEs show good agreement. In particular, two models are more consistent when Knudsen numbers are relatively small, i.e., $\varepsilon=10^{-2},10^{-3}$. Notice that near the contact discontinuity the heat flux is bigger than near the shock, mainly because of the larger gradient of temperature at the contact. We note that the discrepancy between two solutions for $\varepsilon=10^{-1}$ are bigger than the other cases, which again confirms the difference of kinetic and fluid models in the rarefied regime.

\begin{figure}[h]
	\centering
	\begin{subfigure}[h]{0.49\linewidth}
		\includegraphics[width=1\linewidth]{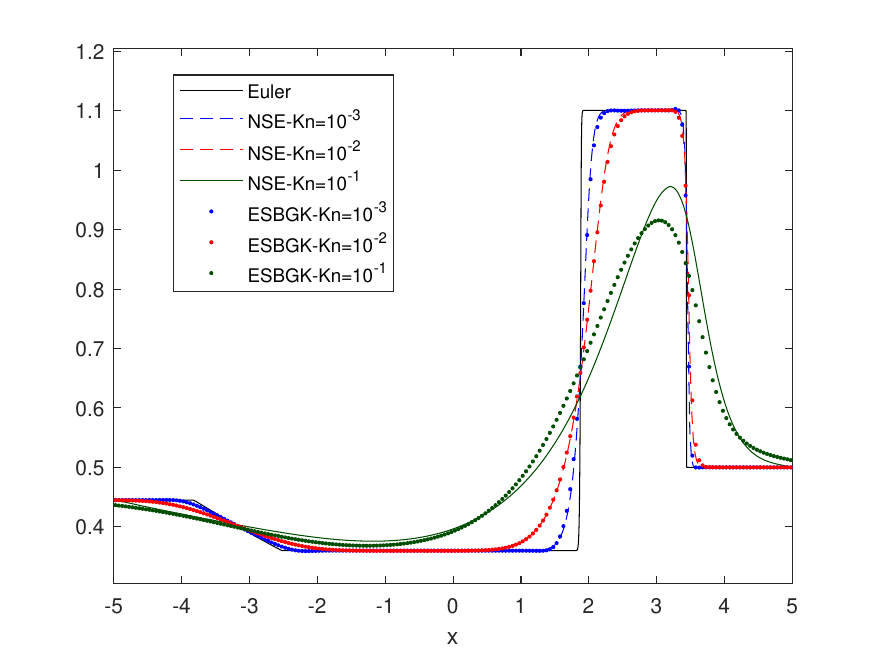}
		\subcaption{$\rho$}
	\end{subfigure}	
	\begin{subfigure}[h]{0.49\linewidth}
		\includegraphics[width=1\linewidth]{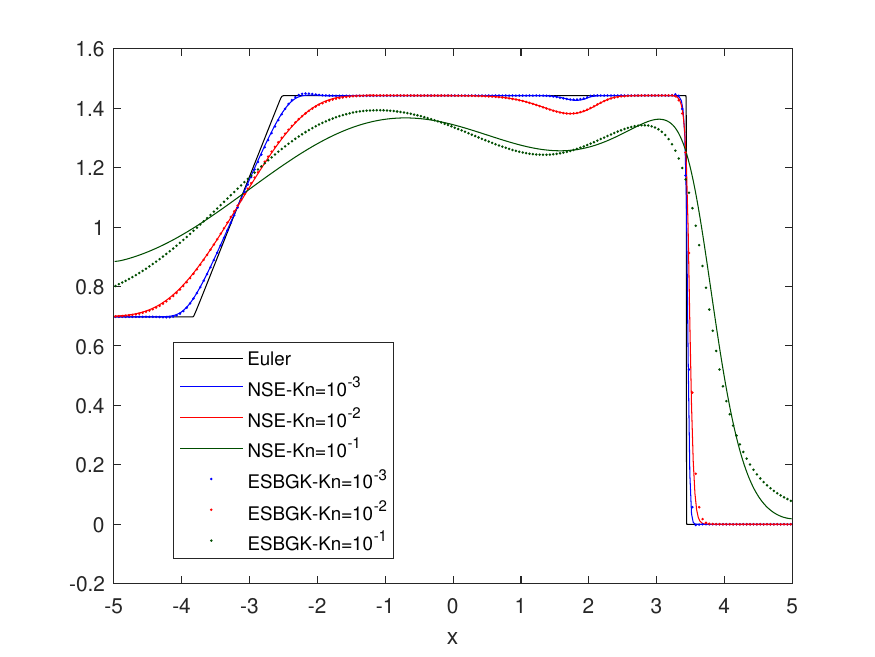}
		\subcaption{$u_1$}
	\end{subfigure}	
	\begin{subfigure}[h]{0.49\linewidth}
		\includegraphics[width=1\linewidth]{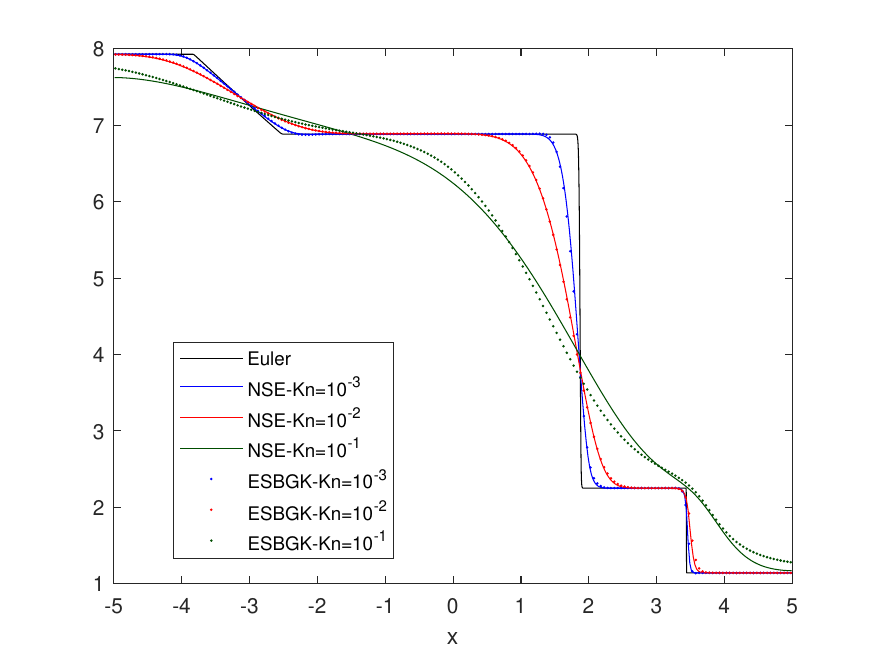}
		\subcaption{$T$}
	\end{subfigure}	
%	\begin{subfigure}[h]{0.32\linewidth}
%		\includegraphics[width=1\linewidth]{figures/Lax_comparison_sigma3d_Nx400_CFL2_first10_123}
%		\subcaption{$\sigma(u)$}
%	\end{subfigure}	
	\begin{subfigure}[h]{0.49\linewidth}
		\includegraphics[width=1\linewidth]{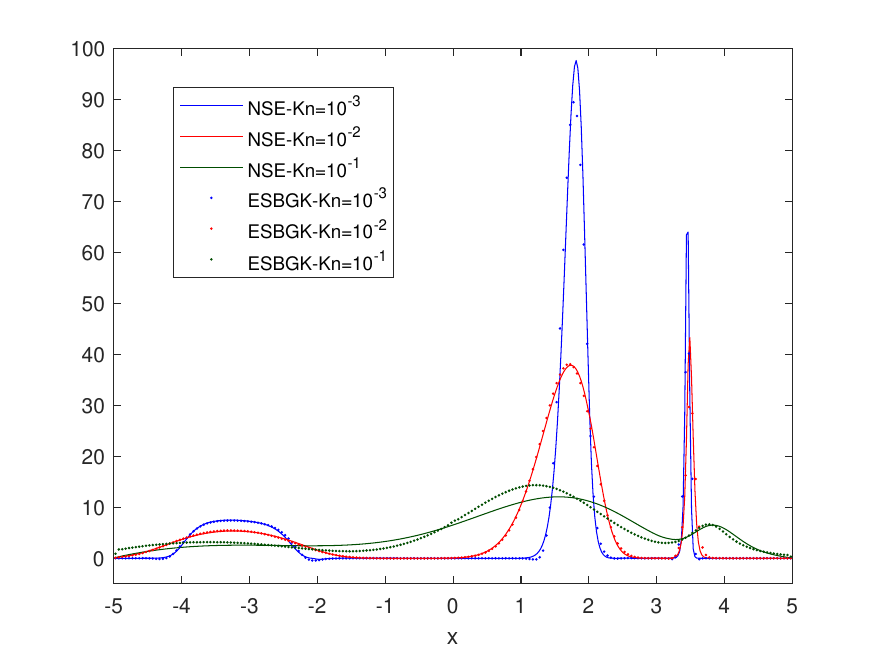}
		\subcaption{$Q$}
	\end{subfigure}	
	
	\caption{Lax shock tube problem.}\label{fig lax2}
\end{figure}

\section{Conclusion}
In this work, we have developed a class of semi-Lagrangian methods for the ES-BGK model of the BTE. By employing the semi-Lagrangian approach, we successfully circumvent the time step restrictions imposed by the convection term. For the treatment of the nonlinear stiff relaxation operator at small Knudsen numbers, we employed high-order $L$-stable DIRK or BDF methods. Notably, the proposed implicit schemes are designed to update solutions explicitly without the need for a Newton solver. Conservation has been attained by adopting a conservative reconstruction and weighted $L^2$-projection techniques. We provide conservation estimates and formal proof for the asymptotic limit of distribution function towards local Maxwellian for high order methods. Moreover, under restrictions on $f$ and $\mathcal{M}$, we show the consistency of first order time discretization between ES-BGK model and NSEs in the semi-Lagrangian framework. Through various numerical tests, we demonstrated the accuracy and efficiency of our methods.\\

\noindent{\bf Acknowledgement:}
This work has been partially supported by the Spoke 10 Future AI Research (FAIR) of the Italian Research Center funded by the Ministry of University and Research as part of the National Recovery and Resilience Plan (PNRR).  G. R. and S. B. would like to thank the Italian Ministry of University and Research (MUR) for the support of this research with funds coming from PRIN Project 2022 (N.\ 2022KA3JBA entitled  ``Advanced numerical methods for time dependent parametric partial differential equations with applications''), and from the project Spoke 1 ``FutureHPC \& BigData'' of the Italian Research Center on High-Performance Computing, Big Data and Quantum Computing(ICSC). S. B. acknowledge partial support from Italian Ministerial grant PRIN 2022 PNRR ``FIN4GEO: Forward and Inverse Numerical Modelling of hydrothermal systems in volcanic regions with application to geothermal energy exploitation.'', (No. P2022BNB97). G. Russo and S. Boscarino are members of the INdAM Research group GNCS. S. Y. Cho was supported by the National Research Foundation of Korea (NRF) grant funded by the Korea government (MSIT) (No. RS-2022-00166144). S.-B. Yun was supported by the National Research Foundation of Korea(NRF) grant funded by the Korean goverment(MSIT) (RS-2023-NR076676).\newline
%%%%%%%%%%%%%%%%%%%%%%%%%%%%%%%%%%%%%%%%%%%%%%%%%%%%%%%%%%%%%%%%%%%%%%%%%%%%%%%%%%%%%%%%%%%%%%%%%%%%%%%%%%%%%%%%%%%%%%%%%%%%%%%%%%%%%%%%%%%%%%%%%%%%%%%%%%%%%%%%%%%%%%%%%%%%%%%%%%%%%%%%%%%%%%%%%%%%%%%

\noindent{\bf Declaration of competing interest}\\
The authors declare that they have no known competing financial interests or personal relationships that could have appeared to influence the work reported in this paper.\newline

\noindent{\bf Data Availibility Statement}\\
No data was used for the research described in the article.

\bibliographystyle{amsplain}

\end{document}